\documentclass{amsart}

\newtheorem{theorem}{Theorem}[section]
\newtheorem{lemma}[theorem]{Lemma}

\theoremstyle{definition}

\theoremstyle{remark}
\newtheorem{remark}[theorem]{Remark}

\numberwithin{equation}{section}

\usepackage{geometry}
\usepackage{graphicx}
\usepackage{epstopdf}
\usepackage{amscd}
\usepackage{color}

  \newcommand{\p}{p}
  \newcommand{\up}{\Pzk \p}
  \newcommand{\pI}{p_I}
  \newcommand{\q}{q}
  \newcommand{\qstar}{\q^*}
  \newcommand{\barp}{\overline{\p}}
  \newcommand{\barq}{\overline{\q}}
  \newcommand{\ph}{p_h}
   \newcommand{\qh}{\q_h}
  \newcommand{\qstarh}{\q^*_h}
  \newcommand{\psiI}{\psi_I}

\newcommand{\wiz}{\text{\tiny$\aleph$}}

\newcommand{\bb}{{\bf b}}
\newcommand{\diffp}{{ \kappa}}
\newcommand{\trasp}{{\bf b}}
\newcommand{\vbeta}{\boldsymbol{\beta}}
\newcommand{\diffh}{\ph-\pI}
\newcommand{\qdiff}{q^*_h}
\newcommand{\reaction}{\gamma}

  \newcommand{\E}{E}
  
  \newcommand{\Th}{{\mathcal T}_h}
  \newcommand{\Eh}{{\mathcal E}_h}

  \newcommand{\xx}{{\bf x}}

  \newcommand{\ds}{{\rm d}s}
  
 \newcommand{\dx}{{\rm d}x}

\newcommand{\RR}{ {\mathbb R}}

    \newcommand{\Pp}{\mathbb P}
    \newcommand{\Vhnk}{{{\mathcal Q}_h^k}}
     \newcommand{\Vhnkt}{{{\widetilde{\mathcal Q}}_h^k}}

    \newcommand{\Amf}{{\mathfrak L}}
    \newcommand{\Amfs}{{\mathfrak L}^*}
    \newcommand{\curl}{\operatorname{\bf curl}}
    \renewcommand{\div}{\operatorname{div}}

\newcommand{\GG}{{\mathcal G}}

     \newcommand{\NErr}{{\mathcal E}}
      
     \newcommand{\Op}{{\mathcal O}}

     \newcommand{\Id}{I}

    \newcommand{\Pzk}{{\Pi}^0_{k}}
    \newcommand{\Pzkm}{\Pi^0_{k-1}}
    \newcommand{\pLdue}{{\Pi}^0_{k}}
    \newcommand{\pLduem}{\Pi^0_{k-1}}
    \newcommand{\PNk}{{\Pi^\nabla_k}}
    \newcommand{\PGk}{{\Pi^\GG_k}}

   \newcommand{\intE}{\int_\E}

\begin{document}

\title[Virtual Elements for elliptic problems]{Virtual Element Methods for general second order elliptic problems on polygonal meshes}

\author{L. Beir\~ao da Veiga}
\address{Dipartimento di Matematica, Universit\`a di Milano,
Via Saldini 50, 20133 Milano (Italy),
and IMATI del CNR, Via Ferrata 1, 27100 Pavia, (Italy).}
\email{lourenco.beirao@unimi.it}

\author{F. Brezzi}
\address{IUSS, Piazza della Vittoria 15, 27100 Pavia (Italy),
and IMATI del CNR, Via Ferrata 1, 27100 Pavia, (Italy).}
\email{brezzi@imati.cnr.it}

\author{L.D. Marini}
\address{Dipartimento di Matematica, Universit\`a di Pavia,
and IMATI del CNR, Via Ferrata 1, 27100 Pavia, (Italy).}
\email{marini@imati.cnr.it}

\author{A. Russo}
\address{Dipartimento di Matematica e Applicazioni, Universit\`a di Milano-Bicocca,
via Cozzi 57, 20125 Milano (Italy)
and
IMATI del CNR, Via Ferrata 1, 27100 Pavia, (Italy).}
\email{alessandro.russo@unimib.it}

\subjclass[2010]{65N30}

\date{}

\dedicatory{}

\begin{abstract}
We consider the discretization of a boundary value problem for a general linear second-order elliptic operator
with smooth coefficients using the Virtual Element approach. As in \cite{Schatz74} the problem is supposed to have a unique solution, but the associated bilinear form is not supposed to be coercive. Contrary to what was previously done
for Virtual Element Methods (as for instance in \cite{volley}), we use here, in a systematic way, the
$L^2$-projection operators as designed in \cite{projectors}. In particular, the present method {\it does not} reduce
to the original Virtual Element Method of \cite{volley} for simpler problems as the classical Laplace operator
(apart from the lowest order cases). Numerical experiments show the accuracy and the robustness
of the method, and  they show as well  that a simple-minded extension of the method 
in \cite{volley} to the case of variable coefficients produces, in general, sub-optimal results.

\end{abstract}

\maketitle

\section{Introduction}\label{introduction}
The aim of this paper is to design and analyze the use of  Virtual Element Methods (in short, VEM) for the approximate solution of general linear second order elliptic problems in two dimensions. In particular we shall deal with diffusion-convection-reaction problems with variable coefficients.

For the simpler case of Laplace operator in two dimensions the Virtual Element Method in the primal form (see \cite{volley})  could be seen essentially as a re-formulation (in a simpler, more elegant and easier to analyze manner) of the Mimetic Finite Difference method as presented in \cite{Brezzi:Buffa:Lipnikov:2009}  for the lowest order case, and extended to arbitrary order in \cite{BLM11}.

Actually, in more recent times both Mimetic Finite Differences and Virtual Element Methods have been growing very fast, allowing a much wider type of discretizations (arbitrary degree, arbitrary
continuity, nonconforming or discontinuous variants) as well as different types of applications. See in particular, for Mimetic Finite Differences, \cite{Antonietti},
 \cite{MFD7}, \cite{BLM09}, \cite{BeiraodaVeiga:Manzini:2008b}, 
\cite{Brezzi:Lipnikov:Shashkov:2005}, \cite{Brezzi:Lipnikov:Shashkov:Simoncini:2007}, \cite{Brezzi:Lipnikov:Simoncini:2005}, \cite{Brezzi:Lipnikov:Shashkov:2006},
and mostly \cite{MFD23}, \cite{MFD22} (and the references therein),
and
\cite{projectors}, \cite{ABMV14},
\cite{Berrone-VEM}, \cite{VEM-elasticity}, \cite{super-misti}, \cite{BM13},
\cite{Brezzi:Marini:plates}, \cite{hourglass}, \cite{Gain-PhD}, \cite{Paulino-VEM}, \cite{VEM19}, \cite{VemSteklov}, \cite{POLY37} for Virtual Elements.

We point out, on the other hand, that the use of polygonal and polyhedral meshes for the approximate solution of Partial Differential Equations, but also for several other branches of Scientific Computing,
is surely not reduced to Mimetic Finite Differences or Virtual Element Methods.  Indeed,
polygonal (and then polyhedral) decompositions have already a long story, and often are
based on approaches that are substantially  different from MFDs or VEMs. We recall for instance \cite{Arroyo-Ortiz}, \cite{MESHLESS12}, \cite{GFEM13}, 
\cite{MESHLESS16}, \cite{GFEM17}, \cite{Bishop}, \cite{Bochev-Hyman},  
\cite{Cockburn-Jay-Lazarov},
\cite{DiPietro-Ern}, \cite{Droniou:Eymard:Gallouet:Herbin:2010},
\cite{GFEM62}, \cite{Gillette-2}, \cite{FHK06}, \cite{Fries:Belytschko:2010}
\cite{XFEM84}, \cite{MESHLESS24}, \cite{POLY28}, \cite{LIBROXFEM-128},
  \cite{Gillette-1},  \cite{GFEM146}, \cite{Sukumar:Malsch:2006},
\cite{CRACKXFEM180}, \cite{ST04}, \cite{Tabarraei:Sukumar:2007},  
\cite{TPPM10}, \cite{Wachspress75},  \cite{Wachspress11}, \cite{POLY47}.

Most of these methods use trial and test functions of a rather complicate nature, that often
could be computed (and integrated) only in some approximate way. The same is (even more) true for Virtual Element Methods where trial and test functions are solutions of  PDE problems inside each element. However, these local problems are not solved, not even in an approximate way,
and the general idea is (roughly speaking) to try to compute  {\it exactly} the values of the local (stiffness) bilinear form when one of the two entries is of polynomial type, and then stabilize the rest, in a rather brutal way. Keeping this in mind, it is clear that for Virtual Element Methods the extension from the constant coefficients to the variable ones is less trivial than for other methods, and in particular, simple minded approaches to variable coefficients can lead to a loss of optimality, especially for higher order methods, as we show with numerical evidence at the end of this paper.

In more recent times several other methods for polygonal decompositions have been introduced in which the trial and test functions are {\it pairs of polynomial} (instead of a single non-polynomial function). See  \cite{Bonelle:Ern:2014}, \cite{Cockburn-IMU}, \cite{Cockburn-Jay-Lazarov}, \cite{Cockburn-Guzman-HWang}, , \cite{DiPietro-Ern-1}, \cite{DiPietro-Ern-2},
\cite{DiPietro-Ern-3}, \cite{DiPietro-Ern--Lemaire}, \cite{Droniou-gradient}, \cite{Mu-Wang-Wei-Ye-Zhao}, \cite{Mu-Wang-Ye},
\cite{Mu-Wang-Ye-2}, \cite{Wang-1},
\cite{Wang-2}.
Though different, these methods surely have many points in common with each other, and with Virtual Element Methods. The main difference is that in the Virtual Element Methods we have indeed, on each element, both boundary  and internal degrees of freedom, but they refer to {\it the same function} (as it is normal for traditional Finite Element Methods), that however is not a polynomial, while in these other methods we have {\it two different functions} that are both polynomials.

However we could consider that the internal degrees of freedom refer to a different (polynomial) function, that has the same moments as the VEM one (as it is done for instance in Mimetic Finite Differences, where the degrees of freedom are treated more as {\it co-chains} rather than values attached to a specific function).  In this respect, the relationships among all these methods definitely  deserves a deeper analysis.

The most recent Virtual Element approach  (already  hinted in \cite{projectors} for dealing with Laplace operator in three dimensions and later extended to mixed formulations in \cite{super-misti}) consists in a tricky way to make the $L^2$-projection operator computable in an exact way starting from the degrees of freedom, with the idea to use, as often as possible, the $L^2$-projection of test and trial functions in place of the functions themselves.

A question that often arises when presenting Virtual Element approximations is: "Since the approximate solution is not explicitly known inside the elements, how can it be represented? And/or how can we  compute its value at points of interest that are internal to elements?" What we suggest is simply to use the $L^2-$projection of the VEM-solution onto piecewise polynomials of degree $k$. In Section \ref{Num-Exp} we provide numerical results  showing the general behavior of the error, and also the error in some internal point following this path.

An outline of the paper is as follows. After stating the problem and its formal adjoint in Section \ref{problem}, we recall in Section \ref{primal} the variational formulation. Then, in Section \ref{VEM-primal} we introduce the Virtual Element approximation. Section \ref{sec:pri:est} is devoted to prove optimal error estimates in $H^1$ and in $L^2$, given in Theorem \ref{risultato-primal} and Theorem \ref{stima-L2-primal}, respectively. Finally, numerical results are presented in Section \ref{Num-Exp}.

Throughout the paper we will use the standard notation $(\cdot\,,\,\cdot)$ or $(\cdot\,,\,\cdot)_0$ to indicate the $L^2$ scalar product. Whenever confusion may arise, we will underline the domain explicitly; for instance $(\cdot\,,\,\cdot)_{0,\E}$ will denote the $L^2(\E)$ scalar product on a generic polygon $\E$.
For every geometrical object $\Op$ and for every integer $k\ge -1$ we denote by $\Pp_k(\Op)$ the set of polynomials of degree $\le k$ on $\Op$, with $\Pp_{-1}(\Op)\equiv \{0\}$, as usual. Whenever no confusion may arise, we will simply use $\Pp_k$, without declaring explicitly the domain.

\section{The problem and the adjoint problem}\label{problem}

Let $\Omega\subset \RR^2$ be a bounded convex polygonal domain with boundary $\Gamma$, let  $\diffp$  and $\reaction$ be smooth functions $\Omega\rightarrow\RR$  with $\diffp(\xx)\ge \diffp_0>0$ for all $\xx\in\Omega$, and let $\bb$ be a smooth vector valued function $\Omega\rightarrow\RR^2$. In the sequel  $\diffp_{\max}, \reaction_{\max}$ and $b_{\max}$ will denote the ($L^{\infty}-$like) norm  of the coefficients $\diffp, \reaction, \bb$, respectively.

Assume that the problem
\begin{equation}\label{Pb-cont}
\left\{
\begin{aligned}
\Amf\,\p:=\div(-\diffp(\xx)\nabla\p + \bb(\xx)\p) + \reaction(\xx)\,\p
&= f(\xx)\quad\text{in }\Omega\\
\p &= 0\quad\text{on }\Gamma\\
\end{aligned}
\right.
\end{equation}
is solvable for any $f\in H^{-1}(\Omega)$, and that the estimates
\begin{equation}\label{bound11}
\|\p\|_{1,\Omega} \le C \|f\|_{-1,\Omega}
\end{equation}
and
\begin{equation}\label{bound02}
\|\p\|_{2,\Omega} \le C \|f\|_{0,\Omega}
\end{equation}
hold with a constant $C$ independent of $f$. We point out that  these assumptions imply, among other things, that
existence and uniqueness hold, as well, for the (formal) adjoint operator $\Amfs$ given by
\begin{equation}\label{def:adjoint}
\Amfs\p:=\div(-\diffp(\xx)\nabla\p)- \bb(\xx)\cdot\nabla\p + \reaction(\xx)\,\p.
\end{equation}
Moreover, for every $g\in L^2(\Omega)$ there exists a unique $\varphi\in H^2(\Omega)\cap H^1_0(\Omega)$ such that $\Amfs\varphi=g$, and
\begin{equation}\label{stimadj}
\|\varphi\|_{2,\Omega} \le C^* \|g\|_{0,\Omega}
\end{equation}
for a constant $C^*$ independent of $g$.
As we shall see, the 2-regularity \eqref{bound02} and \eqref{stimadj} is not strictly necessary in order to
get the results of the present work, and an $s$-regularity with $s>1$ would be sufficient. Here however we are not interested in minimizing the regularity assumptions.

We also point out that the choice of having a scalar diffusion coefficient was done just for simplicity. Having a full diffusion tensor would not change the analysis in a substantial way. Actually, in the numerical results presented in Section \ref{Num-Exp} a full tensor is used.

\section{ Variational formulation}\label{primal}

Set:
\begin{equation}\label{tre-uno}
a(\p,\q):=\int_{\Omega}\diffp \nabla \p\cdot \nabla \q\,{\rm d}x,\quad b(\p,\q):=-\int_{\Omega} \p (\trasp\cdot\nabla \q) \,{\rm d}x,\quad c(\p,\q):=\int_{\Omega}\reaction \p \,\q\,{\rm d}x
\end{equation}
and define
\begin{equation}\label{tre-due}
B(\p,\q):=a(\p,\q)+b(\p,\q)+c(\p,\q).
\end{equation}
The variational formulation of problem \eqref{Pb-cont} is
\begin{equation}
\left \lbrace{
\begin{aligned}\label{Pb-var}
&\mbox{Find } \p\in  H^1_0(\Omega) \mbox{ such that}\\
&B(\p,\q)=(f,\q)\quad \forall \q\in H^1_0(\Omega).
\end{aligned}
} \right.
\end{equation}

\begin{remark}\label{exuniq}
It is immediate to check that our assumptions on the coefficients imply that the bilinear form $B(\cdot,\cdot)$ verifies
\begin{equation}\label{cont-B}
B(\p,\q)\le M \|\p\|_1 \|q\|_1, \quad \p,\q \in H^1(\Omega)
\end{equation}
and hence
$$
\|\Amf\p\|_{-1}=\sup_{\q\in H^1_0} \frac{<\Amf\p,\q>}{\|\q\|_1}=\sup_{\q\in H^1_0}\frac{ B(\p,\q)}
{\|\q\|_1}\le M \| \p \|_1 .
$$
It is also easy to check that this, together with \eqref{bound11}, implies that
\begin{equation}\label{ellipt-B}
 \sup_{\q\in H^1_0} \frac{B(\p,\q)}{\|\q\|_1}\ge C_B \|\p\|_1 \quad \forall \p \in H^1_0(\Omega),
\end{equation}
for some constant $C_B>0$ independent of $\p$.
%
%
On the other hand it is also well known that \eqref{cont-B} and \eqref{ellipt-B} imply existence and uniqueness of the solution of problem \eqref{Pb-var}.

\end{remark}

\section{VEM approximation}\label{VEM-primal}

In the present section we introduce the virtual element discretization of \eqref{Pb-var}.

\subsection{The Virtual Element space}

Let $\Th$ be a decomposition of $\Omega$ into star-shaped polygons $\E$, and let $\Eh$ be the set of edges $e$ of $\Th$. 

\begin{remark}\label{rem:mesh}
To be precise, we assume that (i) every element $\E$ is star-shaped with respect to every point of a disk $D_{\rho}$ of radius $\rho^{\E} h_E$ (where $h_{\E}$ is the diameter of $\E$), and (ii) that every edge $e$ of $\E$ has lenght $|e|\ge \rho^{\E} h_\E$. The first assumption could be relaxed in order to allow unions of star-shaped elements and the second one could be essentially avoided; since such technical generalizations are beyond the scope of the present work, we prefer to keep the simpler conditions stated above. When considering a {\it sequence} of decompositions $\{\Th\}_h$ we will obviously assume $\rho^{\E}\ge\rho_0>0$ for some $\rho_0$ independent of $\E$ and of the decomposition. As usual, $h$ will denote the maximum diameter of the elements of $\Th$.
\end{remark}

Following \cite{volley,projectors}, for every  integer $k \ge 1$ and for every element $\E$ we start by defining a \emph{preliminary} local space:
\begin{equation}\label{def-VhE-prel}
\Vhnkt (\E):=\{\q\in H^1(\E):~\q_{|e} \in \Pp_k(e) ~\forall e\in \partial \E,~\Delta \q \in \Pp_{k}(\E)\}.
\end{equation}
On $\Vhnkt (\E)$ the following set of linear operators are well defined. For all $\q \in \Vhnkt(\E)$:
\begin{itemize}
\item[($D_1$)] the values $q(V_i)$ at the vertices $V_i$ of $\E$,
\end{itemize}
and for $k\ge 2$
\begin{itemize}
\item[($D_2$)] the edge moments $\int_{e}\q\,\p_{k-2}\,\ds$, $\p_{k-2}\in \Pp_{k-2}(e)$, on each edge $e$ of $\E$,
\item[($D_3$)] the internal moments $\int_{\E}\q \,\p_{k-2}\,\dx$, $\p_{k-2}\in \Pp_{k-2}(\E)$.
\end{itemize}
We point out that for each element $\E$ and for all $k$ the operators $D_1$--$D_3$ satisfy the following property:
\begin{equation}\label{uni-big}
\{q\in\Pp_k(\E)\} \mbox{ and }\{D_i(q)=0, i=1,2,3\} \mbox{  imply }  \{q=0\}.
\end{equation}
Property \eqref{uni-big} implies that on each element $\E$ we can easily construct  a projection
operator from  $\Vhnkt$ to $\Pp_k$ that depends only on
$D_1$--$D_3$ and is explicitly computable starting from them. Let us see how.  Let
$n_V$ be the number of vertices of $E$, and let
\begin{equation}\label{dimD}
n_D:=n_V k+k(k-1)/2
\end{equation}
be the ``cardinality'' of $D_1$--$D_3$ (with obvious meaning). Consider the mapping
from $\Vhnkt(\E)$ to $\RR^{n_D}$ defined by $Dq:=(D_1$--$D_3)(q)$, and choose
a bilinear symmetric positive form $\GG$ on $\RR^{n_D}\times\RR^{n_D}$ (for
instance, the Euclidean scalar product on $\RR^{n_D}$). For every $q\in\Vhnkt(\E)$
we define $\PGk q\in\Pp_k$ as the unique solution of
\begin{equation}\label{defG}
\GG(Dq-D\PGk q,Dz)=0\qquad\forall\, z \in\, \Pp_k.
\end{equation}
It is obvious that $\PGk q_k\equiv q_k$ for every $q_k\in\Pp_k$, and also that
$\PGk q$ depends only on the values of $D_1 q$, $D_2 q$, and $D_3 q$. It can be rather
easily proved that {\it every projection operator $\Vhnkt(\E)\rightarrow\RR^{n_D}$ depending  only on the values of $D_1$--$D_3$ can be obtained by \eqref{defG} for a suitable choice
of the bilinear form $\GG$}. It is also obvious that collecting all the local projection
operators we can construct every global projection operator from $\Vhnkt$ to the space
of piecewise $\Pp_k$ functions.

Here however (both for {\it historical} reasons and for convenience of computation) we
will focus our attention on a particular choice of projection operator. For this we  recall from \cite{volley,projectors} the definition of the operator $\PNk$:
for any $q \in H^1_0(\Omega)$, the function $\PNk q$ on each element $\E$ is a polynomial in $\Pp_k(\E)$, defined by
\begin{equation}\label{defPina}
(\nabla(\PNk \q-\q),\nabla p_k)_{0,\E}=0\quad\mbox{and}\quad \int_{\partial\E}(\PNk \q-\q)\ds=0
\qquad\forall p_k\in\Pp_k.
\end{equation}
This operator is well defined on $\Vhnkt (\E)$ and, most important, for all $\q \in \Vhnkt (\E)$ the polynomial $\PNk\q$ can be computed using only the values of the operators (D) calculated on $\q$. This follows easily with an integration by parts, see for instance \cite{volley}.

We are now ready to introduce our local Virtual space
\begin{equation}\label{def-VhE}
\Vhnk (\E):=\{\q\in \Vhnkt (\E) \: : \: \int_\E \q\,\p_k \,\dx= \int_\E (\PNk\q) \p_k \,\dx \ \ \forall \p_k \in (\Pp_k/\Pp_{k-2}(\E))\} ,
\end{equation}
where the space $\big(\Pp_k/\Pp_{k-2}(\E)\big)$ denotes the polynomials in $\Pp_k(\E)$ that are $L^2(\E)$ orthogonal to $\Pp_{k-2}(\E)$.
The corresponding global space is:
\begin{equation}\label{def-Vh}
\Vhnk :=\{\q\in H^1_0(\Omega):~\q_{|\E} \in \Vhnk(\E)~\forall \E\in \Th\}.
\end{equation}
Let now $\Pzk$ denote the $L^2-$ projection onto $\Pp_k$, defined locally, as usual, by
\begin{equation}\label{projections}
(\q-\pLdue \q, p_k)_{0,\E} =0 \qquad \forall p_k\in\Pp_k .
\end{equation}
For simplicity of notation, in the following we will denote by the same symbol
also the $L^2-$ projection of vector valued functions onto the polynomial space $[\Pp_k]^2$.

We note that it can be proved, see again \cite{volley,projectors} that the set of linear operators $(D)$ are a {set of degrees of freedom} for the virtual space $\Vhnk(\E)$.

Clearly the degrees of freedom $(D)$ define an {\it interpolation operator} that associates to each
smooth enough function $\varphi$ its {\it interpolant} $\varphi_I\in\Vhnk(\E)$ that shares
with $\varphi$ the values of the degrees of freedom.
Moreover, the virtual space $\Vhnk(\E)$ satisfies the following four properties:
\begin{itemize}
\item $\Pp_k(\E) \subseteq \Vhnk(\E)$ (trivial to check);
\item for all $q\in\Vhnk(\E)$, the function $\PNk q$ can be explicitly computed from the degrees of freedom $(D)$ of $q$ (see \cite{volley,projectors});
\item for all $q\in\Vhnk(\E)$, the function $\Pzk q$ can be explicitly computed from the degrees of freedom $(D)$ of $q$ (see \cite{projectors});
\item for all $q\in\Vhnk(\E)$, the vector function $\Pzkm \nabla q$ can be explicitly computed from the degrees of freedom $(D)$ of $q$ (see below).
\end{itemize}
While the second and third properties above can be found in the literature, and thus are not detailed here, we need to spend some words on the last one. In order to compute $\Pzkm \nabla q$, for all $\E\in\Th$ we must be able to calculate
$$
\int_\E \nabla q \cdot {\bf p}_{k-1}\,{\rm d}x \quad 
\
\forall {\bf p}_{k-1} \in [\Pp_{k-1}(\E)]^2 .
$$
An integration by parts, denoting by ${\bf n}$ the outward unit normal to the element boundary $\partial\E$, gives
$$
\int_\E \nabla q \cdot {\bf p}_{k-1}\,{\rm d}x = - \int_\E q \, \textrm{div} ({\bf p}_{k-1})\,{\rm d}x
+ \int_{\partial\E} q \, ({\bf p}_{k-1}\cdot{\bf n}) \,\ds .
$$
The first term in the right hand side above clearly depends only on the moments of $q$ appearing in $(D_3)$. The second term can also be computed since $q$ is a polynomial of degree $k$ on each edge and therefore $q_{|\partial\E}$ is uniquely determined by the values of $(D_1)$ and $(D_2)$.
Needless to say, all the above properties extend in an obvious way to the global space \eqref{def-Vh}. In particular, we point out that, for a smooth function $\varphi \in H^1_0(\Omega)$, its global interpolant $\varphi_I$ is in $Q^k_h$.


We end this section by showing some simple bounds on the operator $\PNk$.
Applying \eqref{defPina} for $p_k=\PNk \q$ we have
$$\|\nabla\PNk \q\|_{0,\E}^2=(\nabla \q, \nabla\PNk \q)_{0,\E}\le \|\nabla \q\|_{0,\E}\,\|\nabla
 \PNk \q\|_{0,\E}$$
giving  immediately
\begin{equation}\label{boundPN}
\|\nabla\PNk \q\|_{0,\E}\le \|\nabla \q\|_{0,\E}.
\end{equation}
Moreover, always from the definition \eqref{defPina},
$$(\nabla(\q-\PNk\q),\nabla(\q-\PNk\q))_{0,\E}=(\nabla(\q-\PNk\q),\nabla \q)_{0,\E}\le
|\q-\PNk\q|_{1,\E} |\nabla\q|_{0,\E}$$
that immediately gives
\begin{equation}\label{boundPN1}
|\q-\PNk \q|_{1,\E}\le |\q|_{1,\E} .
\end{equation}
Finally, using again  the definition \eqref{defPina} we have
\begin{equation*}
\|\nabla (\q- \PNk\q)\|^2_{0,\E}=( \nabla(\q-\PNk\q),\nabla(\q-\pLdue\q) )_{0,\E}
\le \|\nabla (\q-\PNk\q)\|_{0,\E}\|\nabla(\q-\pLdue\q) \|_{0,\E},
\end{equation*}
giving
\begin{equation}\label{spostata}
\|\nabla (\q-\PNk\q)\|_{0,\E}\le \|\nabla(\q-\pLdue\q) \|_{0,\E}.
\end{equation}

\subsection{The discrete problem}

We now introduce the discrete bilinear forms that will be used in the method. Since we will mostly work on a generic element $\E$, we will denote by $a^{\E}(\cdot,\cdot), b^{\E}(\cdot,\cdot), c^{\E}(\cdot,\cdot),$ and $B^{\E}(\cdot,\cdot)$ the restriction to $\E$ of the corresponding bilinear forms defined in \eqref{tre-uno}-\eqref{tre-due}.
Let $S^{\E}(\p,\q)$ be a symmetric bilinear form on $\Vhnk(\E) \times \Vhnk(\E)$
that scales like $a^{\E}(\cdot, \cdot)$ on the kernel of $\PNk$. More precisely, we assume that  $\exists \,\alpha_*,\alpha^*$ independent of $h$ with $0<\alpha_*\le \alpha^*$ such that
\begin{equation}
\alpha_* a^{\E}(\qh,\qh) \le S^{\E}(\qh,\qh)\le \alpha^* a^{\E}(\qh,\qh)
\quad \forall \qh \in \Vhnk(\E) \textrm{ with } \PNk\qh=0.
\end{equation}
Examples on how to build the bilinear form above can be found in \cite{volley,hitchhikers}.
Note that, due to the symmetry of $S^{\E}$, this implies, for all $\ph,\qh \in \Vhnk(\E)$ with $\PNk\ph=\PNk\qh=0$,
\begin{equation}\label{cont-S}
S^{\E}(\ph,\qh)\le (S^{\E}(\ph,\ph))^{1/2}(S^{\E}(\qh,\qh))^{1/2}\le \alpha^* (a^{\E}(\ph,\ph))^{1/2}(a^{\E}(\qh,\qh))^{1/2}.
\end{equation}
We can now define, on each element $\E\in\Th$ and for every $\p$, $\q$ in $\Vhnk(\E)$, the local forms and loading term:
\begin{equation}\label{def:forms}
\begin{aligned}
&a^{\E}_h(\p,\q):= \intE \diffp [\Pzkm\nabla \p]\cdot[\Pzkm\nabla \q]\,\dx
+ S^{\E}((\Id-\PNk)\p,(\Id-\PNk)\q) \\
& b^{\E}_h(\p,\q):=-\intE [\pLduem \p]\,[\trasp \cdot \Pzkm \nabla \q]\,\dx,\\
&c^{\E}_h(\p,\q):=\intE \reaction [\pLduem \p]\,[\pLduem \q]\,\dx, \quad (f_h, \q)_{\E}:=  \int_{\E}  f\; \pLduem \q \, \dx,\\
&B^{\E}_h(\p,\q):=a^{\E}_h(\p,\q)+b^{\E}_h(\p,\q)+c^{\E}_h(\p,\q).
\end{aligned}
\end{equation}
We just recall that, since $\PNk$ is a projection, then
\begin{equation}\label{4.15}
S^{\E}
((\Id-\PNk)p_k,(\Id-\PNk)\q)
=0
\quad\forall p_k\in\Pp_k,\;\forall \q\in \Vhnk(\E),
\end{equation}
and thus, since $S^{\E}$ is symmetric, the $S^\E$ term will vanish whenever one of the two entries of $a^{\E}_h(\cdot,\cdot)$ is a polynomial in $\Pp_k$.

\noindent
Then we set for all $\p,\q \in \Vhnk$
\begin{equation*}
\begin{aligned}
&a_h(\p,\q):=\sum_{\E} a^{\E}_h(\p,\q) , \quad
b_h(\p,\q):=\sum_{\E}b^{\E}_h(\p,\q), \\
& c_h(\p,\q):=\sum_{\E} c^{\E}_h(\p,\q), \quad (f_h, \q):=  \sum_{\E}(f_h, \q)_{\E},\\
\end{aligned}
\end{equation*}
and
\begin{equation}\label{defB}
B_h(\p,\q):=a_h(\p,\q)+b_h(\p,\q)+c_h(\p,\q)= \sum_{\E} B^{\E}_h(\p,\q).
\end{equation}
The approximate problem is:
\begin{equation}
\left \lbrace{
\begin{aligned}\label{vem-approx-n}
&\mbox{Find } \ph\in \Vhnk \mbox{ such that}\\
&B_h(\ph,\q)=(f_h,\q)\quad \forall \q\in \Vhnk.
\end{aligned}
} \right.
\end{equation}

\begin{remark}
The bilinear forms $b_h^\E$ and $c_h^\E$ in \eqref{def:forms} are well defined for all $\p,\q \in H^1(\E)$, as well as the global forms $b_h$ and $c_h$, which are well defined on the whole $H^1_0(\Omega)$.
This does not hold for $a_h^\E$, due to the presence of the stabilizing term $S^\E$ that is defined only on $\Vhnk(\E)$.
\end{remark}

\begin{remark}\label{rem-sm}
We recall that the choice indicated in \cite{volley} would have suggested to define
\begin{equation} \label{simple-minded}
a^{\E}_h(\p,\q):= \intE \diffp [\nabla\PNk \p]\cdot[\nabla \PNk\q]\,\dx
+ S^{\E}((\Id-\PNk)\p,(\Id-\PNk)\q) .
\end{equation}
Actually, it can be easily seen that for $k=1$ this coincides with our choice \eqref{def:forms}. This is not the case for $k\ge 2$. In particular, a deeper analysis shows heavy losses in the order of convergence for $k\ge 3$. In Section \ref{Num-Exp} we provide an example for $k= 4$.
On the other hand, it can be shown that if $\diffp \nabla \p$ happens to be a
gradient the choice \eqref{simple-minded} does work.
\end{remark}
\section{Error estimates}\label{sec:pri:est}

In the present section we derive error estimates for the proposed method.

\subsection{Preliminary results}

We now present some preliminary results useful in the sequel. We start by the following approximation lemma, that mainly comes from the mesh regularity assumptions in Remark \ref{rem:mesh} and standard approximation results on polygonal domains (see for instance \cite{volley,VemSteklov}).

Here and in the sequel $C$ will denote a generic positive constant independent of $h$, with different meaning in different occurrencies, and generally depending on the coefficients of the operator $\Amf$. Whenever needed to better follow the steps of the proofs, for a smooth scalar or vector-valued function $\wiz$, we shall use $C_{\wiz}$  to denote a constant depending on $\wiz$ and possibly on its derivatives up to the needed order.

\begin{lemma}
There exists a positive constant $C = C(\rho_0,k)$ such that, for all $\E$ in $\Th$ and all smooth enough functions $\varphi$ defined on $\E$, it holds
$$
\begin{aligned}
& \| \varphi - \pLdue\varphi \|_{m,\E} \le C h_\E^{s-m} | \varphi |_{s,\E}
\qquad m,s \in \mathbb{N}, \  m \le s \le k+1, \\
& \| \varphi - \PNk\varphi \|_{m,\E} \le C h_\E^{s-m} | \varphi |_{s,\E} ,
\qquad m,s \in \mathbb{N}, \  m \le s \le k+1, \ s \ge 1 ,\\
& \| \varphi - \varphi_I \|_{m,\E} \le C h_\E^{s-m} | \varphi |_{s,\E} ,
\qquad m,s \in \mathbb{N}, \  m \le s \le k+1, \ s \ge 2 . \\
\end{aligned}
$$
\end{lemma}
We also have the following continuity lemma.

\begin{lemma}\label{cont-primal}
The bilinear form $B_h(\cdot,\cdot)$ is continuous in $\Vhnk \times \Vhnk$, that is,
\begin{equation}\label{cont-Bh}
 B_h(\p,\q)\le C_{\diffp,\trasp,\reaction} \|\p\|_1\|\q\|_1 \quad \p,\q \in \Vhnk ,
\end{equation}
with $C_{\diffp,\trasp,\reaction}$ a positive constant depending on $\diffp,\trasp,\reaction$ but independent of $h$.
\end{lemma}

\begin{proof}
The continuity of $b_h$ and $c_h$ is obvious, and actually holds on the whole $H^1_0(\Omega)$ space. We have
\begin{equation}\label{cont-bhch}
b_h(\p,\q)\le b_{\max} \|\p\|_0 |\q|_1,\quad c_h(\p,\q)\le \reaction_{\max} \|\p\|_0\|\q\|_0,\quad \p,\q \in H^1_0(\Omega) .
\end{equation}
The continuity of $a_h$ is proved upon observing that, thanks to \eqref{cont-S} and \eqref{boundPN1},
\begin{equation}\label{X-6}
\begin{aligned}
S^{\E}((\Id-\PNk)\p,(\Id-\PNk)\q))&\le \alpha^* \diffp_{\max} |\p-\PNk\p|_{1,\E}|\q-\PNk\q|_{1,\E}\\
&\le \alpha^* \diffp_{\max} |\p|_{1,\E}|\q|_{1,\E}.
\end{aligned}
\end{equation}
Thus:
\begin{equation}\label{cont-ah}
a_h(\p,\q)\le (1+\alpha^*) \diffp_{\max}|\p|_1 |\q|_1 \quad \p,\q \in \Vhnk ,
\end{equation}
and the result follows.
\end{proof}
In many occasions we will need to estimate the difference between continuous and discrete bilinear forms. This is done once and for all in the following preliminary Lemma.

\begin{lemma}\label{general}
Let $\E\in\Th$, let $\mu$ be a smooth function on $\E$, and let $\p,\q$ denote smooth scalar or vector-valued functions  on $\E$.
For a generic $\varphi \in L^2(\E)$  (or in $(L^2(\E))^2$)
we define
\begin{equation}
\NErr^k_\E(\varphi):=\|\varphi-\Pzk{\varphi}\|_{0,\E}.
\end{equation} Then we have the estimate:
\begin{equation}\label{gen1}
(\mu \p,\q)_{0,\E}-(\mu \Pzk \p,\Pzk\q)_{0,\E}\le \NErr^k_\E(\mu \p)\NErr^k_\E(\q)+\NErr^k_\E(\mu \q)\NErr^k_\E(\p)\\
+C_{\mu}\NErr^k_\E(\p )\NErr^k_\E(\q),
\end{equation}
where  $C_{\mu}$ is a constant depending on $\mu$.
\end{lemma}
\begin{proof}
For simplifying the notation we will set $\barp:=\Pzk \p,~\barq:=\Pzk \q$.
By adding and subtracting terms, and by the definition of projection we have
\begin{equation}\label{gen1-bis}
\begin{aligned}
(\mu \p,\q)_{0,\E}-&(\mu \barp,\barq)_{0,\E}=(\mu \p,\q-\barq)_{0,\E}+(\p-\barp,\mu\barq)_{0,\E}\\
&=(\mu \p- \overline{\mu\p},\q-\barq)_{0,\E}+(\p-\barp,\mu \barq-\overline{\mu \q})_{0,\E}\\
&=(\mu \p- \overline{\mu\p},\q-\barq)_{0,\E}+(\p-\barp,\mu \barq-\overline{\mu \q}+\mu\q-\mu\q)_{0,\E}\\
&=(\mu \p- \overline{\mu\p},\q-\barq)_{0,\E}+(\p-\barp,\mu\q-\overline{\mu \q})_{0,\E}-(\p-\barp,\mu(\q- \barq))_{0,\E},\\
\end{aligned}
\end{equation}
and  the result follows by Cauchy-Schwarz inequality with $C_\mu=\|\mu\|_{\infty}$.
\end{proof}

The following result follows immediately by a direct application of Lemma \ref{general}.
\begin{lemma}\label{Bh-B}
For all $E\in\Th$ it holds
\begin{equation}\label{ah-a}
\begin{aligned}
a^{\E}_h(\p,\q) -&a^{\E}(\p,\q)\le \NErr_\E^{k-1}(\diffp \nabla\p)\NErr_\E^{k-1}(\nabla\q)
+\NErr_\E^{k-1}(\diffp\nabla \q)\NErr_\E^{k-1}(\nabla\p) \\
&+C_{\diffp}\NErr_\E^{k-1}(\nabla\p )\NErr_\E^{k-1}(\nabla\q) \\
&+ S^{\E}((\Id-\PNk)\p,(\Id-\PNk)\q)) \quad \forall \p,\q \in \Vhnk(E),
\end{aligned}
\end{equation}
\begin{equation}\label{bh-b}
\begin{aligned}
b^{\E}_h(\p,\q) -&b^{\E}(\p,\q)\le \NErr_\E^{k-1}(\bb \cdot\nabla\q)\NErr_\E^{k-1}(\p)+
\NErr_\E^{k-1}(\nabla\q) \NErr_\E^{k-1}(\bb\p)\\
&+C_{\bb}\NErr_\E^{k-1}(\nabla\q )\NErr_\E^{k-1}(\p) \quad \forall \p,\q \in H^1(E),
\end{aligned}
\end{equation}
\begin{equation}\label{ch-c}
\begin{aligned}
c^{\E}_h(\p,\q) -&c^{\E}(\p,\q)\le \NErr_\E^{k-1}(\reaction \p)\NErr_\E^{k-1}(\q)+\NErr_\E^{k-1}(\reaction \q)\NErr_\E^{k-1}(\p)\\
&+C_{\reaction}\NErr_\E^{k-1}(\p )\NErr_\E^{k-1}(\q) \quad \forall \p,\q \in H^1(E).
\end{aligned}
\end{equation}
\end{lemma}

In the next Lemma we evaluate the consistency error.

\begin{lemma}[{\bf consistency}]
\label{lemma-cons}
For all $\p$ sufficiently regular and for all $\qh \in \Vhnk$ it holds
\begin{equation}\label{consistency-primal}
B^{\E}(\up,\qh) - B^{\E}_h(\up,\qh)
\le C_{\diffp,\trasp,\reaction} h^{k}_{\E}  \|\p\|_{k+1,\E}\|\qh\|_{1,\E} \quad \forall \E\in\Th.
\end{equation}
\end{lemma}
\begin{proof}
From the definition of $B^{\E}$ and $B^{\E}_h$ we have
\begin{equation}\label{X-3}
\begin{aligned}
B^{\E}(\up,\qh)-& B^{\E}_h(\up,\qh)=a^{\E}(\up,\qh)-a^{\E}_h(\up,\qh) \\
&+ b^{\E}(\up,\qh)-b^{\E}_h(\up,\qh) + c^{\E}(\up,\qh)-c^{\E}_h(\up,\qh).
\end{aligned}
\end{equation}
We first observe that when $\p\in \Pp_k(\E)$, then obviously we have $\Pzk \, \p \equiv \p$,
$\Pzkm \nabla \p \equiv \nabla \p$, and then by \eqref{4.15} the term containing $S^{\E}$ vanishes. Therefore, a direct application of \eqref{ah-a} implies
\begin{equation}\label{uffa4}
a^{\E}_h(\up,\qh) -a^{\E}(\up,\qh) \le \NErr_\E^{k-1}(\diffp \nabla \up)\,\NErr_\E^{k-1}(\nabla\qh),
\end{equation}
for all $\qh\in\Vhnk(E)$.
The first factor in the right-hand side of \eqref{uffa4} can be easily bounded by
\begin{equation}\label{X-1}
\begin{aligned}
\NErr_\E^{k-1}(\diffp \nabla \up) & = \|  \diffp \nabla \up - \Pzkm (\diffp \nabla \up) \|_{0,\E}
\le  \|  \diffp \nabla \up - \Pzkm (\diffp \nabla \p) \|_{0,\E} \\
& \le  \|  \diffp \nabla \up - \diffp \nabla \p \|_{0,\E} +
 \|  \diffp \nabla \p - \Pzkm (\diffp \nabla \p) \|_{0,\E} \\
& \le C\, h^{k}_{\E} ({\diffp_{\max}} |\p|_{k+1,\E}+|\diffp \nabla \p|_{k,\E})\le C_{\diffp} h^{k}_{\E}  \|\p\|_{k+1,\E},
\end{aligned}
\end{equation}
and the second factor can by simply bounded by $\|\qh\|_{1,\E}$.
Thus,
\begin{equation}\label{uffa4-bis}
a^{\E}_h(\up,\qh) -a^{\E}(\up,\qh) \le C_{\diffp} h^{k}_{\E}  \|\p\|_{k+1,\E} \|\qh\|_{1,\E}.
\end{equation}
With similar arguments we have, for instance,
\begin{equation}\label{XX-2}
\begin{aligned}
\NErr_\E^{k-1}(\trasp \up) &\le C( h^{k+1}_{\E} {\trasp_{\max}} |\p|_{k+1,\E}+h^k_{\E}|\trasp \p|_{k,\E})
\le C_{\trasp} h^{k}_{\E}  \|\p\|_{k+1,\E},\\
\NErr_\E^{k-1}(\reaction \up)& \le C( h^{k+1}_{\E} {\reaction_{\max}} |\p|_{k+1,\E}+h^k_{\E}|\reaction \p|_{k,\E})
\le C_{\reaction} h^{k}_{\E}  \|\p\|_{k+1,\E}.
\end{aligned}
\end{equation}
Consequently,
\begin{equation}\label{uffa4-ter}
\begin{aligned}
b^{\E}_h(\up,\qh) -b^{\E}(\up,\qh) &\le C_{\trasp} h^{k}_{\E}  \|\p\|_{k+1,\E} \|\qh\|_{1,\E},\\
c^{\E}_h(\up,\qh) -c^{\E}(\up,\qh) &\le C_{\reaction} h^{k}_{\E}  \|\p\|_{k+1,\E} \|\qh\|_{1,\E}.
\end{aligned}
\end{equation}
The proof follows by inserting \eqref{uffa4-bis} and \eqref{uffa4-ter} in \eqref{X-3}.

%

\end{proof}

\begin{remark}
We point out that \eqref{consistency-primal} holds for a generic $\qh\in \Vhnk$, for which only $H^1$ regularity can be used. If for instance $\qh=\q_I$, that is, $\qh$ is  the interpolate of a more regular function, \eqref{consistency-primal} can be improved. Indeed, looking e.g. at \eqref{uffa4} we would have
\begin{equation}\label{in-remark6}
\begin{aligned}
\NErr_\E^{k-1}(\nabla \q_I)&=\|\nabla \q_I-\Pzkm \nabla \q_I\|_{0,\E}
\le \|\nabla \q_I-\Pzkm \nabla \q\|_{0,\E}\\
&\le \|\nabla (\q_I- \q)\|_{0,\E}+\|\nabla \q-\Pzkm \nabla \q\|_{0,\E}\le C \, h \|\q\|_{2,\E},
\end{aligned}
\end{equation}
and in \eqref{consistency-primal} we would gain an extra power of $h$:
\begin{equation}\label{cons-interp}
B^{\E}(\up,\q_I) - B^{\E}_h(\up,\q_I) \le C_{\diffp,\trasp,\reaction} h^{k+1}_{\E}  \|\p\|_{k+1,\E}\|\q\|_{2,\E}.
\end{equation}
\end{remark}
Before going to study the error estimates for our problem, we have to prove a final technical Lemma.
\begin{lemma}\label{dual-per-ah}
For every $\qstar\in H^1_0(\Omega)$ there exists a $\qstarh\in \Vhnk$ such that
\begin{equation}\label{defwstar}
a_h(\qstarh,\q_h)=a(\qstar,\q_h)\quad\forall\,\q_h\in \Vhnk.
\end{equation}
Moreover, there exists a constant $C$, independent of $h$, such that
\begin{equation}\label{w-wstar}
h\|\qstar-\qstarh\|_{1,\Omega}+\|\qstar-\qstarh\|_{0,\Omega}\le C \,h\,\|\qstar\|_{1,\Omega}.
\end{equation}
\end{lemma}

\begin{proof} We first remark that, by definition of projection,
we have
\begin{equation}
\|\nabla \q - \Pzkm \nabla \q\|_{0,\E}\le \|\nabla \q - \nabla \PNk \q\|_{0,\E},
\end{equation}
since $\nabla\PNk \q $ is a (vector) polynomial of degree $\le k-1$.
Hence, for $\q \in \Vhnk$ and for every integer $k\ge 1$:
\begin{equation}\label{ellipt-ah}
a_h(\q,\q)\ge C\,
\sum_{\E}\Big(\|\Pzkm \nabla \q\|_{0,\E}^2 + \| (I-\Pzkm)\nabla \q\|_{0,\E}^2)\Big) \ge C |\q|_1^2,
\end{equation}
and this immediately implies that \eqref{defwstar} has a unique solution, and that, using \eqref{cont-ah},
we also have $\|\qstarh\|_1\le C\,\|\qstar\|_1$. In order to show the second part of \eqref{w-wstar}
we shall use duality arguments.
Let $\psi\in H^2(\Omega)\cap H^1_0(\Omega)$ be the solution of
\begin{equation}\label{5.24}
a(\q,\psi)=(\qstar-\qstarh,\q)_{0,\Omega}\quad\forall\q\in H^1_0(\Omega),
\end{equation}
and let $\psiI \in \Vhnk$ be its interpolant, for which it holds
\begin{equation}\label{X-2}
\|\psi-\psiI\|_1 \le C h |\psi|_2 \le C\, h\, \|\qstar-\qstarh\|_0.
\end{equation}
We have easily that, for every $k\ge 0$ (and obvious notation for $\NErr^k$)
$$\NErr^k(\nabla\psiI)\le\NErr^0(\nabla\psiI)\le \NErr^0(\nabla(\psiI-\psi))+\NErr^0(\nabla\psi)
\le\, C\, h \|\psi\|_2 \le\,  C\, h \|\qstar-\qstarh\|_0 ,$$
and similarly
$$
\NErr^k(\diffp\nabla\psiI)\le\NErr^0(\diffp\nabla\psiI)\le\, C_{\diffp}\, h \|\qstar-\qstarh\|_0.
$$
By recalling \eqref{cont-S} and the definition of the projectors, then using standard approximation estimates, we easily get
\begin{equation}\label{stima-S}
\begin{aligned}
S^{\E}((\Id-\PNk)\qstarh,(\Id-\PNk)\psiI)) &
\le \alpha^*\diffp_{\max} \|\nabla \qstarh - \nabla \PNk \qstarh\|_{0,\E}\|\nabla \psiI-\nabla \PNk\psiI \|_{0,\E} \\
& \le \alpha^*\diffp_{\max} |\qstarh|_{1,\E}  \|\nabla \psiI-\nabla \PNk\psi \|_{0,\E} \\
& \le \alpha^*\diffp_{\max} |\qstarh|_{1,\E} ( \|\nabla (\psiI- \psi) \|_{0,\E} + \|\nabla (\psi -  \PNk\psi )\|_{0,\E} ) \\
& \le C \, h_{\E}\, |\qstar|_{1,\E} |\psi|_{2,\E}.
\end{aligned}
\end{equation}
Summation on the elements and \eqref{X-2} give
\begin{equation}\label{5.26bis}
\sum_{\E\in\Th} S^{\E}((\Id-\PNk)\qstarh,(\Id-\PNk)\psiI))\le C \, h\, |\qstar|_{1} \|\qstar-\qstarh\|_{0}.
\end{equation}
On the other hand, both $\NErr^k(\nabla\qstarh)$ and $\NErr^k(\diffp\nabla\qstarh)$ are
just bounded by, say, $C_{\diffp}\|\qstar\|_1$.
Then, using \eqref{5.24}, \eqref{defwstar}, and \eqref{ah-a} (with \eqref{5.26bis} and \eqref{in-remark6}) we obtain
\begin{equation}
\begin{aligned}
\|\qstar-\qstarh\|^2_0& = a(\qstar-\qstarh,\psi)
= a(\qstar-\qstarh,\psi-\psiI)+a(\qstar-\qstarh,\psiI)\\
&= a(\qstar-\qstarh,\psi-\psiI)+a_h(\qstarh,\psiI)-a(\qstarh,\psiI)\\
&\le C_{\diffp}\,\|\qstar-\qstarh\|_1\,\|\psi-\psiI\|_1+C_{\diffp}\|\qstar\|_1\,h\,\|\qstar-\qstarh\|_0,
\end{aligned}
\end{equation}
and the result follows.
\end{proof}

\subsection{$H^1$ Estimate}

We have the following discrete stability lemma.

\begin{lemma}\label{ellipt-primal}
The bilinear form $B_h(\cdot,\cdot)$  satisfies
the following condition (discrete counterpart of \eqref{ellipt-B}): there exists an $h_0>0$ and
a constant $\overline{C}_B$ such that, for all $h<h_0$:
\begin{equation}\label{ellipt-Bh}
 \sup_{\q_h\in \Vhnk} \frac{B_h(\p_h,\q_h)}{\|\q_h\|_1}\ge \overline{C}_B \|\p_h\|_1 \quad\forall\, \p_h \in \Vhnk.
\end{equation}
\end{lemma}
\begin{proof}
In order to prove  \eqref{ellipt-Bh} we follow Schatz \cite{Schatz74}.
For $\ph \in \Vhnk$, from \eqref{ellipt-B} we have
\begin{equation}
\exists \qstar\in H^1_0(\Omega) \mbox{ such that } \frac{B(\ph,\qstar)}{\|\qstar\|_1} \ge C_B \|\ph\|_1.
\end{equation}
Thanks to Lemma \ref{dual-per-ah}, the problem
\begin{equation}
\mbox{Find~} \qstarh\in \Vhnk \mbox{ such that } a_h(\qstarh,v_h)=a(\qstar,v_h)\quad \forall v_h\in \Vhnk
\end{equation}
has a unique solution, that satisfies
\begin{equation}\label{stima-Schatz}
\|\qstarh\|_1 \le C \|\qstar\|_1,\quad \mbox{and } \quad \|\qstar-\qstarh\|_{0,\Omega} \le C\,h \|\qstar\|_1.
\end{equation}
Then,
\begin{equation}
\begin{aligned}
B_h(\ph,\qstarh)&=a_h(\ph,\qstarh) + b_h(\ph,\qstarh) + c_h(\ph,\qstarh)\\
&=a(\ph,\qstar)+ b_h(\ph,\qstarh) - b(\ph,\qstar)+ c_h(\ph,\qstarh)  - c(\ph,\qstar) \\
&+ b(\ph,\qstar) + c(\ph,\qstar) \\
&=B(\ph,\qstar)+ b_h(\ph,\qstarh) - b(\ph,\qstar)+ c_h(\ph,\qstarh)  - c(\ph,\qstar)\\
&=B(\ph,\qstar)+ b(\ph,\qstarh-\qstar) + b_h(\ph,\qstarh) - b(\ph,\qstarh)\\
&+ c_h(\ph,\qstarh-\qstar) + c_h(\ph,\qstar)  - c(\ph,\qstar).
\end{aligned}
\end{equation}
From \eqref{stima-Schatz} and \eqref{cont-bhch} we have
\begin{equation}
c_h(\ph,\qstarh-\qstar) \le \reaction_{\max} h \|\ph\|_0 \|\qstar\|_1,
\end{equation}
while an integration by parts and again \eqref{stima-Schatz} yield
\begin{equation}
\begin{aligned}
b(\ph,\qstarh-\qstar) & = -\int_\Omega \ph \trasp\cdot\nabla (\qstarh-\qstar) \,{\rm d}x
=  \int_\Omega {\rm div} (\trasp\ph) \, (\qstarh-\qstar) \,{\rm d}x \\
& \le \| {\rm div} (\trasp\ph) \|_{0} \,h \| \qstar \|_1 \le C_\trasp^\dagger \| \ph \|_1 \, h \| \qstar \|_1 .
\end{aligned}
\end{equation}
Moreover from \eqref{bh-b} with $\p=\ph,~\q=\qstarh$
\begin{equation}
\begin{aligned}
b_h(\ph,\qstarh) - b(\ph,\qstarh)&\le \NErr^{k-1}(\bb \cdot\nabla\qstarh)\,\NErr^{k-1}(\ph)+\NErr^{k-1}( \nabla \qstarh)\;\NErr^{k-1}(\bb\ph)\\
&+C_{\bb}\NErr^{k-1}(\nabla\qstarh)\,\NErr^{k-1}(\ph)\\
&=\|\bb \cdot\nabla\qstarh-\pLduem (\bb\cdot\nabla \qstarh)\|_0\, \|\ph-\pLduem \ph\|_0\\
&+\|\nabla\qstarh-\Pzkm \nabla \qstarh\|_0 \|\bb\ph-\Pzkm(\bb \ph)\|_0\\
&+C_{\bb}\|\nabla\qstarh -\Pzkm \nabla\qstarh\|_0 \|\ph-\pLduem\ph\|_0\\
&\le \|\bb \cdot\nabla\qstarh\|_0\,C\, h \,|\ph|_1+|\qstarh|_1\, C\,h\, |\bb\ph |_1 +C_{\bb} |\qstarh|_1\,C\, h
|\ph|_1\\
&\le C_{\trasp}^*\,h \|\ph\|_1\|\qstar\|_1.
\end{aligned}
\end{equation}
Similarly, from \eqref{ch-c} we deduce
\begin{equation}
c_h(\ph,\qstar)  - c(\ph,\qstar) \le C_{\reaction}^*\,h \|\ph\|_0\|\qstar\|_1 \le C_{\reaction}^*\,h \|\ph\|_1\|\qstar\|_1 .
\end{equation}
Choosing then $h_0:=\frac{C_B}{2(C_{\trasp}^*+ C_{\reaction}^*+C_\trasp^\dagger+\reaction_{\max})}$
we obviously have for $h\le h_0$,
\begin{equation}\label{per-infsup}
(C_{\trasp}^*+ C_{\reaction}^*+C_\trasp^\dagger+\reaction_{\max})\,h \le \frac{C_B}{2}.
\end{equation}
Hence, for $h \le h_0$,
\begin{equation}
B_h(\ph,\qstarh)\ge  \frac{C_B}{2} \|\ph\|_1 \|\qstarh\|_1,
\end{equation}
and the proof is concluded.
\end{proof}

\begin{remark}
Clearly, if $\trasp=0,$ and $\reaction=0$, \eqref{per-infsup} holds for any $h$ (and, indeed, we are back
at the situation of Lemma  \ref{dual-per-ah}).
\end{remark}

We are now ready to prove the following Theorem.

\begin{theorem}\label{risultato-primal}
For $h$ sufficiently small,
problem \eqref{vem-approx-n} has a unique solution $\ph\in \Vhnk$, and the following error estimate holds:
\begin{equation}\label{stima-finale-primal}
\|\p-\ph\|_1 \le C h^{k}\,(\|\p\|_{k+1}+|f|_{k}),
\end{equation}
with $C$ a constant depending on $\diffp, \vbeta,$ and $\reaction$ but independent of $h$.
\end{theorem}
\begin{proof}
The existence and uniqueness of the solution of problem \eqref{vem-approx-n}, for $h$ small, is a consequence of Lemma \ref{ellipt-primal}. To prove the estimate \eqref{stima-finale-primal},
using \eqref{ellipt-Bh} we have that  for $h\le h_0$ there exists a $\qdiff\in \Vhnk$ verifying
\begin{equation}
\frac{B(\diffh,\qdiff)}{\|\qdiff\|_1} \ge \overline{C}_B \|\diffh\|_1.
\end{equation}
Recalling that $B_h(\ph,\qdiff)=(f_h,\qdiff)$, and $B(\p,\qdiff)=(f,\qdiff)$, adding and subtracting $\pLdue \p$ some simple algebra yields:
\begin{equation}\label{X-4}
\begin{aligned}
{\overline C}_{B} \|\diffh\|_1 \|\qdiff\|_1 &\le B_h(\diffh,\qdiff)=B_h(\ph,\qdiff)-B_h(\pI,\qdiff)\\
& = (f_h,\qdiff) + B_h(\up-\pI,\qdiff) - B_h(\up,\qdiff) + B(\up,\qdiff) \\
&+ B(p-\up,\qdiff) - B(p,\qdiff) \\
& = (f_h-f,\qdiff) + B_h(\up-\pI,\qdiff) + \big( B(\up,\qdiff) - B_h(\up,\qdiff) \big)\\
& + B(p-\up,\qdiff) .
\end{aligned}
\end{equation}
The first term in the right hand side of \eqref{X-4} is bounded by the Cauchy-Schwarz inequality and standard approximation estimates on the load $f$. The second and fourth term are bounded similarly using the continuity of $B_h$ and $B$ combined with approximation estimates for $p$. Finally, the third term is bounded using Lemma \ref{lemma-cons} on each element $\E$.
We get
$$
{\overline C}_{B} \|\diffh\|_1 \|\qdiff\|_1   \le C\, h^k \Big(
 C_{\diffp,\trasp,\reaction}\, \|\p\|_{k+1}+|f|_{k}\Big)\|\qdiff\|_1,
$$
and the proof is concluded.
\end{proof}

\begin{remark}
It is immediate to check that, by the same proof, also the following refined result holds:
$$
\|\p-\ph\|_1 \le C \Big( \sum_{\E \in \Th} h_\E^{2k} \, ( \|\p\|_{k+1,\E}^2+|f|_{k,\E}^2 ) \Big)^{1/2} .
$$
\end{remark}

\subsection{$L^2$ estimate }

We have the following result.

\begin{theorem}\label{stima-L2-primal}
For $h$ sufficiently small, the following error estimate holds:
\begin{equation}\label{stima-finale-primal-0}
\|\p-\ph\|_0 \le C h^{k+1}\,(\|\p\|_{k+1}+|f|_{k}),
\end{equation}
where $C$ is a constant depending on $\diffp, \vbeta,$ and $\reaction$ but independent of $h$.
\end{theorem}
\begin{proof}
Once more, we  shall use duality arguments.
Let $\psi\in H^2(\Omega)\cap H^1_0(\Omega)$ be the solution of the adjoint problem (see \eqref{def:adjoint})
\begin{equation}
\Amf^*\psi=\p-\ph,
\end{equation}
and let $\psiI \in \Vhnk$ be its interpolant, for which it holds
\begin{equation}\label{stima-psi}
\|\psi-\psiI\|_1 \le C h |\psi|_2 \le C h \|\p-\ph\|_0.
\end{equation}
Then:
\begin{equation}
\begin{aligned}
\|\p-\ph\|^2_0& = B(\p-\ph, \psi)= B(\p,\psi-\psiI)+B(\p,\psiI)-B(\ph,\psi)\\
&= B(\p,\psi-\psiI)+(f,\psiI)+B_h(\ph,\psiI)-(f_h,\psiI)-B(\ph,\psi)\\
&=B(\p-\ph,\psi-\psiI)+(f-f_h,\psiI) +B_h(\ph,\psiI)-B(\ph,\psiI)\\
&=B(\p-\ph,\psi-\psiI)+(f-f_h,\psiI-\pLduem \psiI) \\
&+B_h(\ph-\up,\psiI)-B(\ph-\up,\psiI)\\
&+B_h(\up,\psiI)-B(\up,\psiI).
\end{aligned}
\end{equation}
Next:
\begin{equation}
\begin{aligned}
&B(\p-\ph,\psi-\psiI)\le C h^{k+1}\|\p\|_{k+1} \|\p-\ph\|_0,\quad \\
&(f-f_h,\psiI-\pLduem \psiI)\le C h^{k+1} |f|_{k} \|\p-\ph\|_0.
\end{aligned}
\end{equation}
From \eqref{cons-interp} with $\q_I=\psiI$, and \eqref{stima-psi}
\begin{equation}
B_h(\up,\psiI)-B(\up,\psiI)
\le C_{\diffp,\bb,\reaction} h^{k+1}\|\p\|_{k+1} \|\p-\ph\|_0,
\end{equation}
and from \eqref{ah-a}--\eqref{ch-c} with $\p=\ph-\up,~\q=\psiI$, adding and subtracting $\p$,
\begin{equation}
B_h(\ph-\up,\psiI)-B(\ph-\up,\psiI)\le C_{\diffp,\bb,\reaction} h^{k+1}\|\p\|_{k+1} \|\p-\ph\|_0.
\end{equation}
Hence,
\begin{equation}
\|\p-\ph\|_0 \le C_{\diffp,\bb,\reaction} h^{k+1}(\|\p\|_{k+1} + |f|_{k}).
\end{equation}
\end{proof}

\begin{remark} As it can be easily seen from our proofs, the extension to the three-dimensional case would not present major difficulties. We chose to skip it here in order to avoid the use of a  heavier notation and  a certain
amount of technicalities.
\end{remark}

\section{Numerical Experiments}
\label{Num-Exp}

\newcommand{\pex}{p_\textup{ex}}
\newcommand{\xmax}{x_\textup{max}}
\newcommand{\ymax}{y_\textup{max}}

We will consider problem \eqref{Pb-cont} on the unit square with
\begin{equation}
\diffp(x,y) = \begin{pmatrix}y^2+1&-xy\\-xy&x^2+1\end{pmatrix},\quad
\bb=(x,y),\quad
\reaction = x^2+y^3,
\end{equation}
and with right hand side and Dirichlet boundary conditions defined
in such a way that the exact solution is
\begin{equation}
\pex(x,y) := x^2 y + \sin(2\pi x) \sin(2\pi y)+2.
\end{equation}
We will show, in a loglog scale, the convergence curves of the error in $L^2$ and $H^1$
between $\pex$ and the solution $\ph$
given by the Virtual Element Method \eqref{vem-approx-n}.
As the VEM solution $\ph$ is not explicitly known inside the elements, we compare $\pex$ with the $L^2-$projection of $\p_h$ onto $\Pp_k$, that is, with $\Pi^0_k\,\p_h$. We will also show the behaviour of $|\pex -\Pi^0_k\,\p_h|$ at the
maximum point of $\pex$ which is approximately at $(\xmax,\ymax)=(0.781,0.766)$.
%

\subsection{Meshes}

For the convergence test we consider four sequences of meshes.

The first sequence of meshes (labelled \texttt{Lloyd-0}) is a
random Voronoi polygonal tessellation of the unit square
in 25, 100, 400 and 1600 polygons.
The second sequence (labelled \texttt{Lloyd-100})
is obtained starting from the previous one and
performing 100 Lloyd iterations  leading to a Centroidal Voronoi Tessellation (CVT)
(see e.g.  \cite{Du:Faber99}).
%
The 100-polygon mesh of each family is shown in Fig.~\ref{fig:Lloyd-0}
(\texttt{Lloyd-0}) and in Fig.~\ref{fig:Lloyd-100} (\texttt{Lloyd-100})
respectively.
\begin{figure}
\hfill
 \begin{minipage}[b]{0.45\textwidth}
  \begin{center}
  \includegraphics[width=\textwidth]{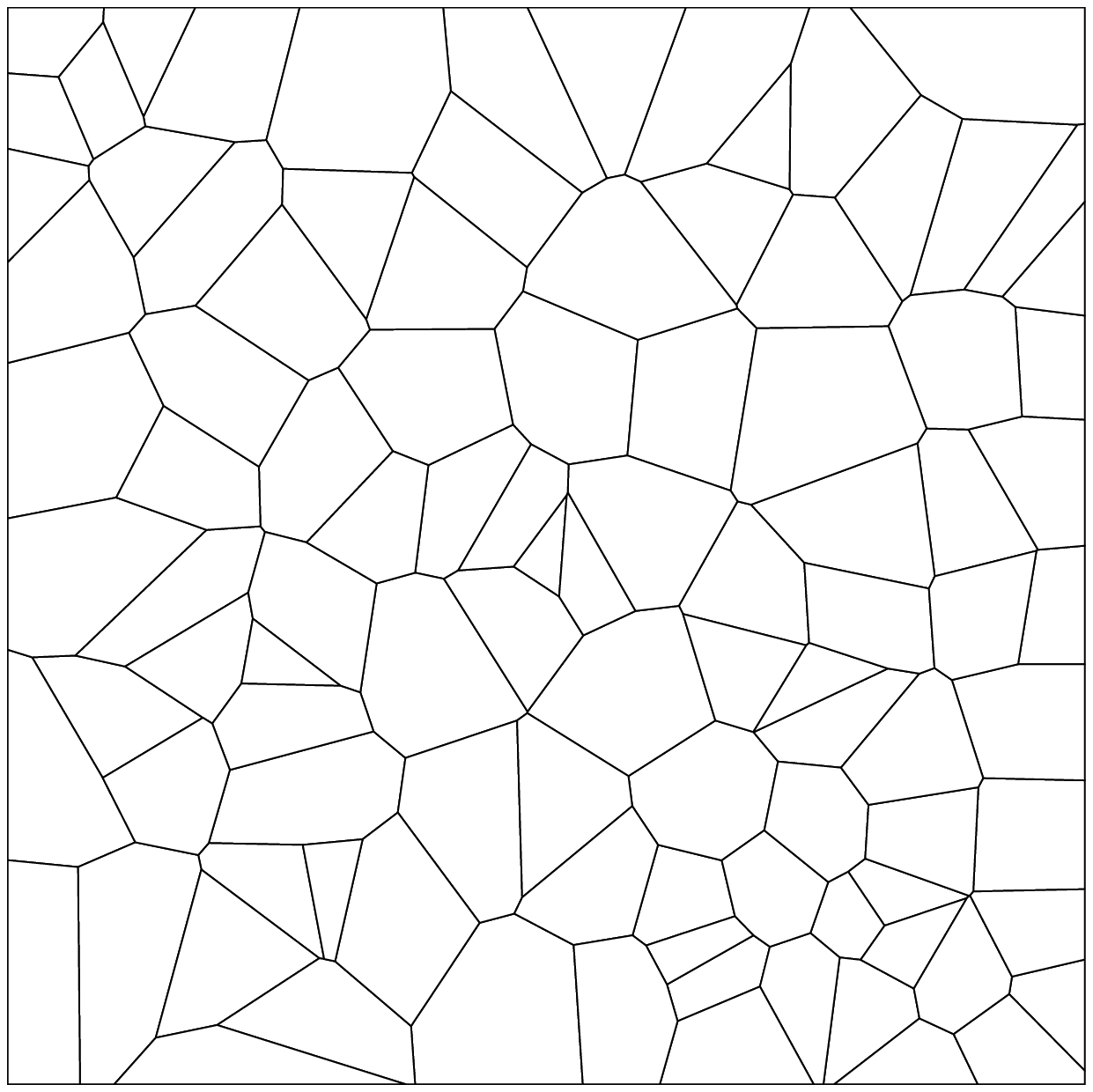}
  \end{center}
  \caption{\texttt{Lloyd-0} mesh}
 \label{fig:Lloyd-0}
 \end{minipage}
\hfill
 \begin{minipage}[b]{0.45\textwidth}
  \begin{center}
  \includegraphics[width=\textwidth]{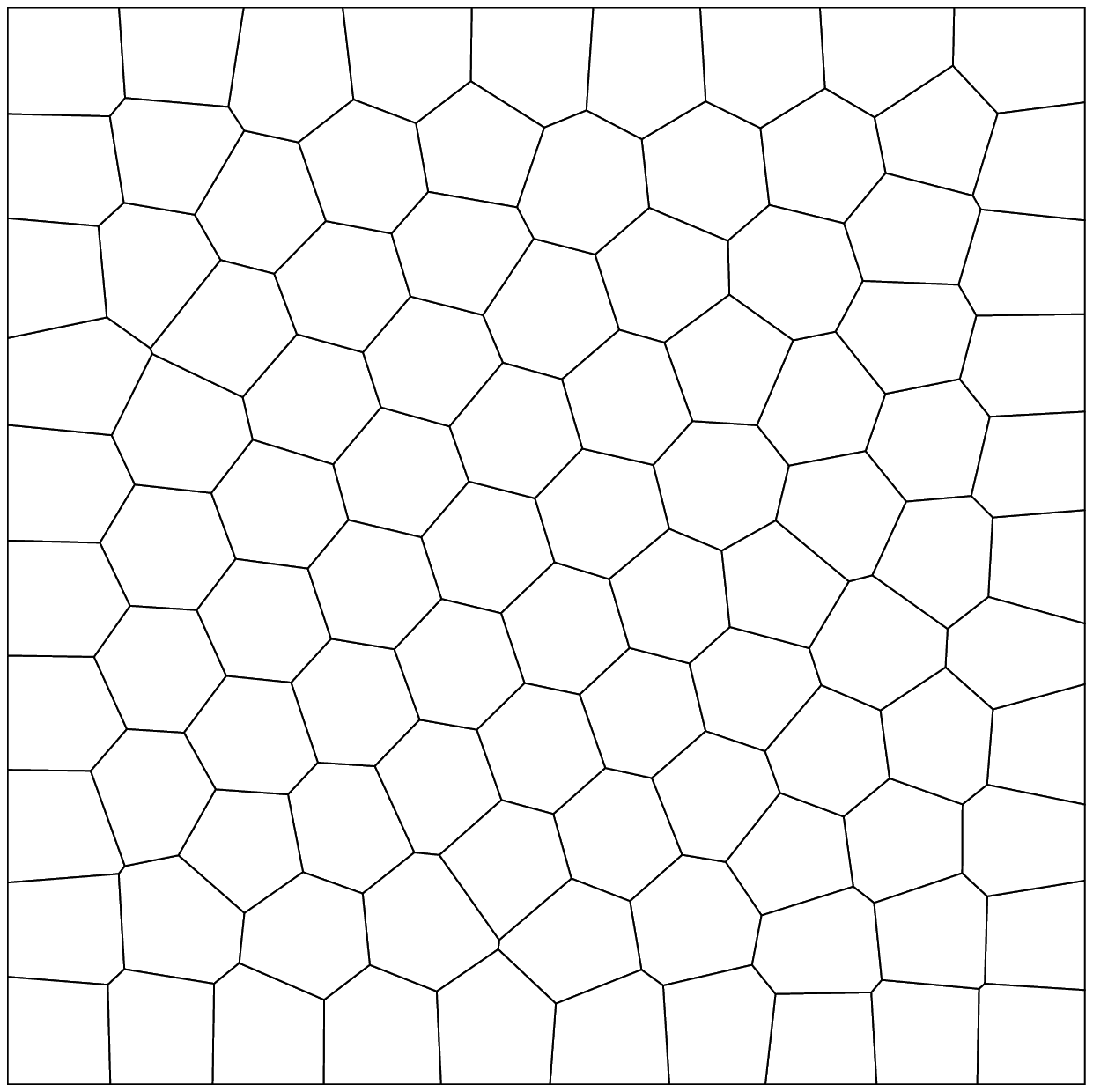}
  \end{center}
  \caption{\texttt{Lloyd-100} mesh}
 \label{fig:Lloyd-100}
 \end{minipage}
\hfill
\end{figure}

The third sequence of meshes (labelled \texttt{square}) is simply a decomposition
of the domain in 25, 100, 400 and 1600 equal squares, while the fourth sequence
(labelled \texttt{concave}) is obtained from the previous one by subdividing
each small square into two non-convex (quite nasty) polygons.
As before, the second meshes of the two sequences are shown in Fig. \ref{fig:square}
and in Fig. \ref{fig:concave} respectively.
\begin{figure}
\hfill
 \begin{minipage}[b]{0.45\textwidth}
  \begin{center}
  \includegraphics[width=\textwidth]{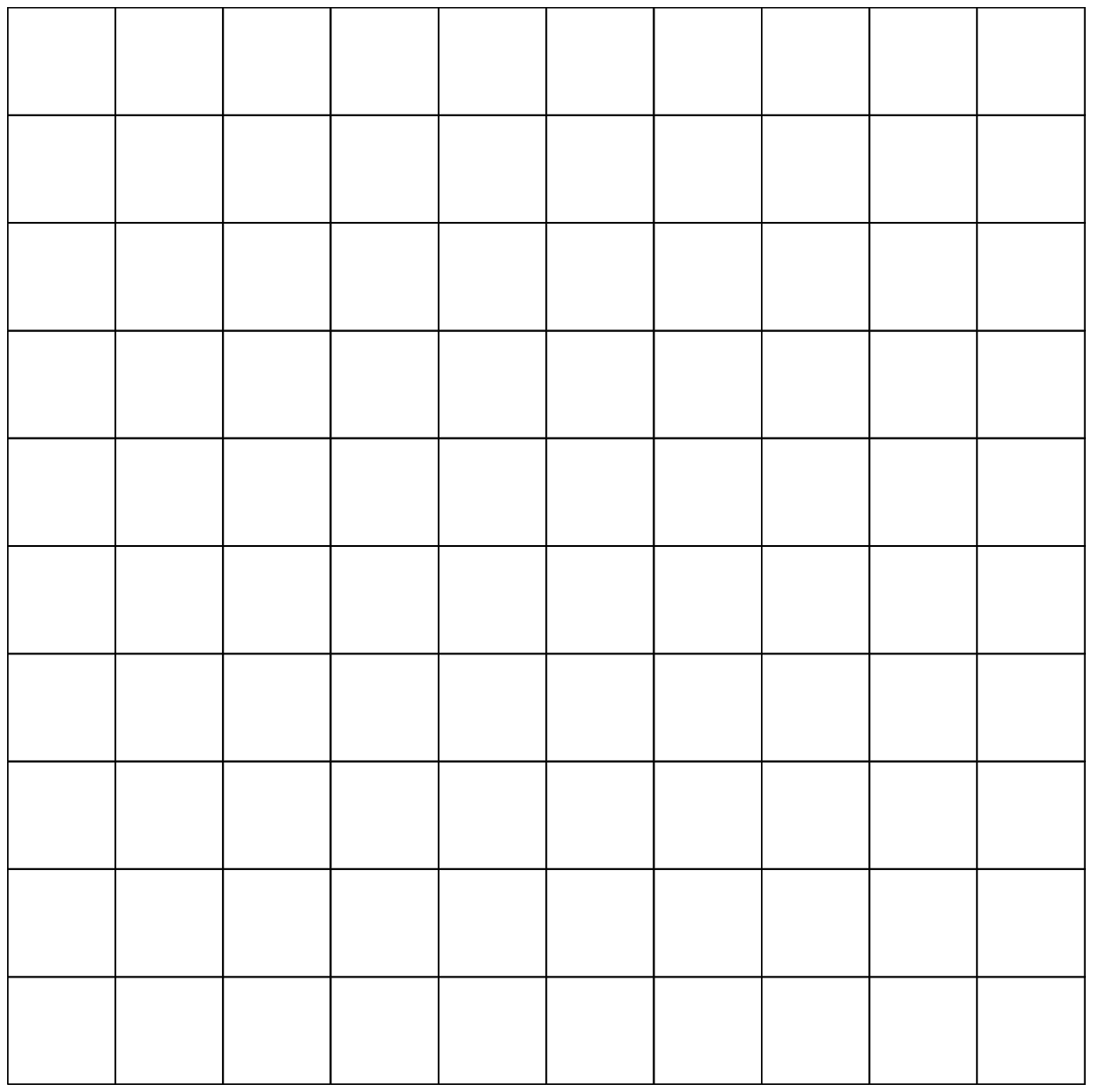}
  \end{center}
  \caption{\texttt{square} mesh}
 \label{fig:square}
 \end{minipage}
\hfill
 \begin{minipage}[b]{0.45\textwidth}
  \begin{center}
  \includegraphics[width=\textwidth]{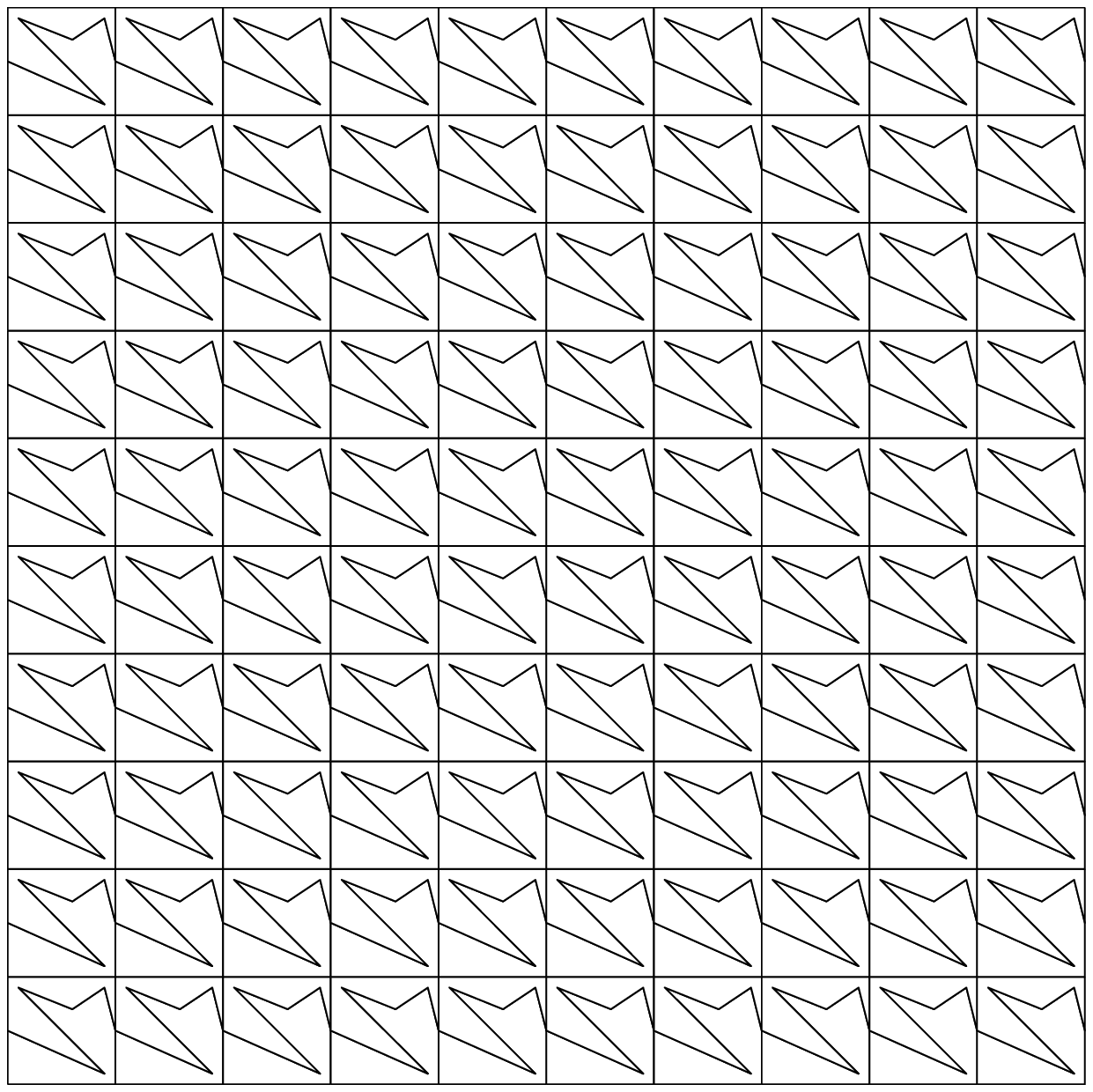}
  \end{center}
  \caption{\texttt{concave} mesh}
 \label{fig:concave}
 \end{minipage}
\hfill
\end{figure}

\subsection{Case $k=1$}

We start to show the convergence results for $k=1$. In Figs. \ref{fig:errL2-k=1-P0km1_nabla} and \ref{fig:errH1-k=1-P0km1_nabla}
we report the relative error in $L^2$ and $H^1$, respectively, for the four mesh sequences.
In Fig. \ref{fig:errp-k=1-P0km1_nabla} we report the relative error at the maximum
point $(\xmax,\ymax)$.
Finally, Fig. \ref{fig:errL2-k=1-nabla_PNk} shows the relative
error in $L^2$ obtained with the method \eqref{simple-minded} (that is, the simple-minded extension of \cite{volley}).
As observed in Remark \ref{rem-sm},
$\Pi^0_k\nabla\equiv\nabla\Pi^\nabla_k$ for $k=1$,
hence the graphs of Fig. \ref{fig:errL2-k=1-P0km1_nabla} and of
Fig. \ref{fig:errL2-k=1-nabla_PNk} are identical.
\begin{figure}
\hfill
 \begin{minipage}[b]{0.49\textwidth}
  \begin{center}
  \includegraphics[width=\textwidth]{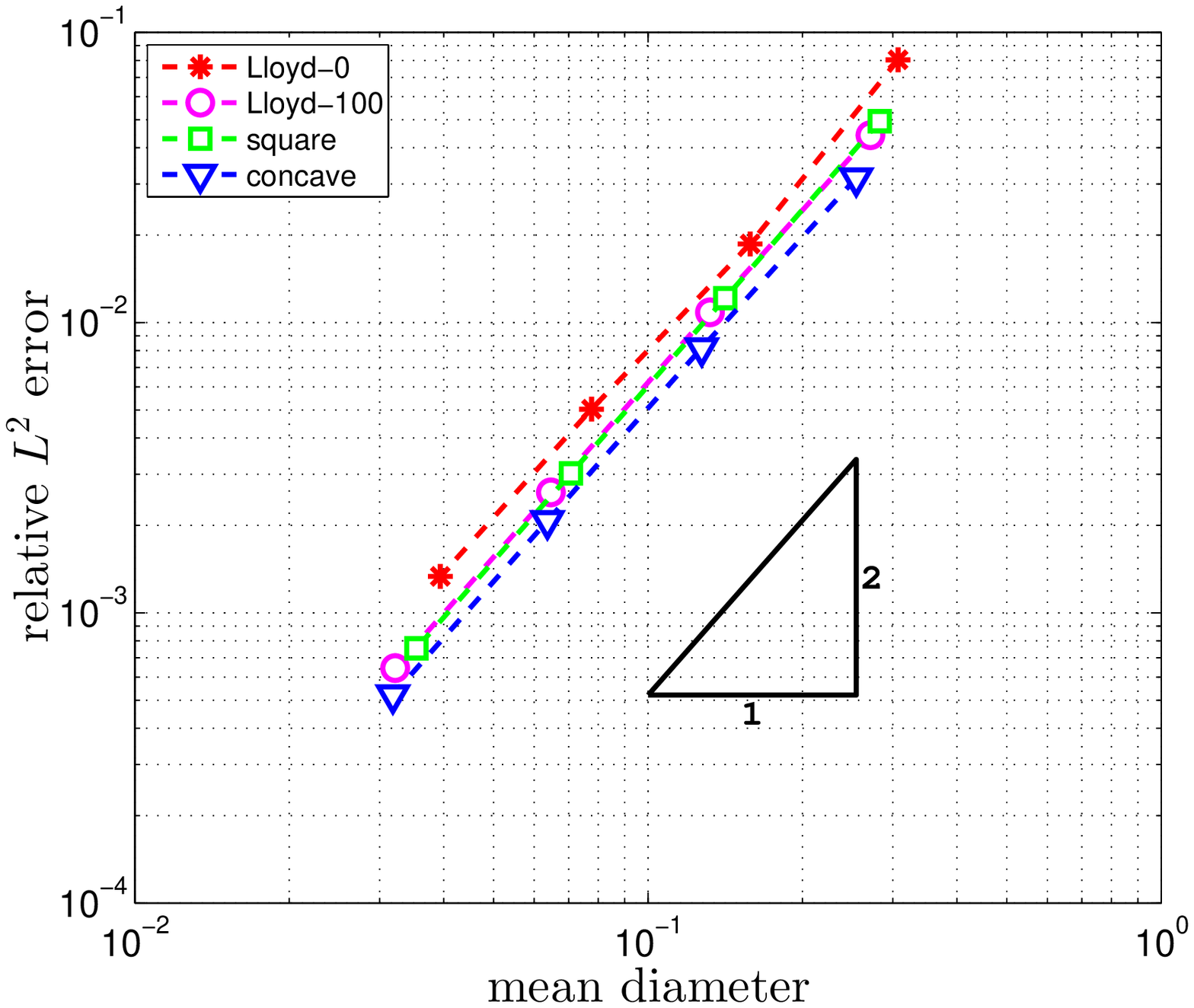}
  \end{center}
  \caption{$k=1$, relative $L^2$ error}
 \label{fig:errL2-k=1-P0km1_nabla}
 \end{minipage}
\hfill
 \begin{minipage}[b]{0.49\textwidth}
  \begin{center}
  \includegraphics[width=\textwidth]{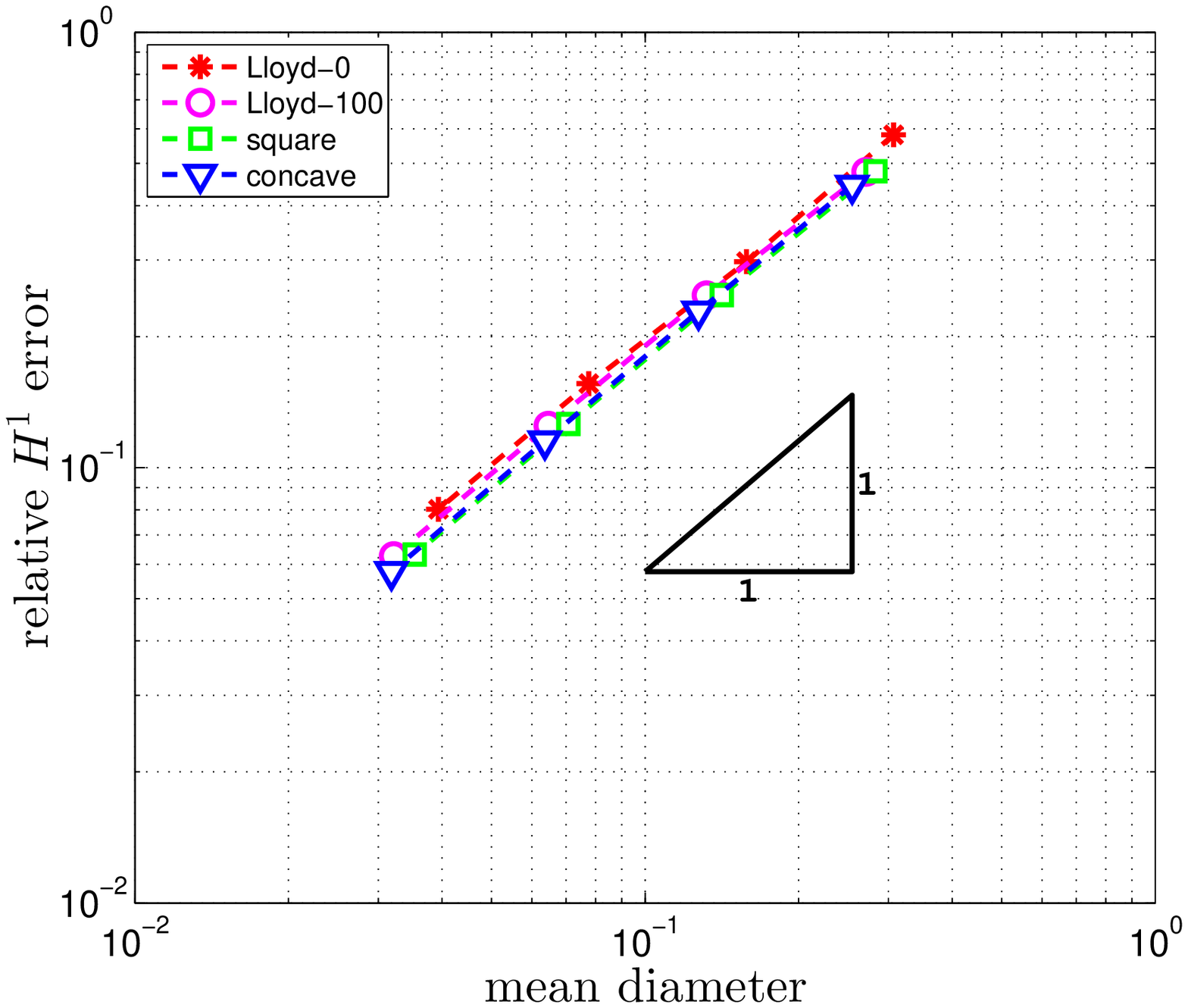}
  \end{center}
  \caption{$k=1$, relative $H^1$ error}
 \label{fig:errH1-k=1-P0km1_nabla}
 \end{minipage}
\hfill
\end{figure}
\begin{figure}
\hfill
 \begin{minipage}[b]{0.49\textwidth}
  \begin{center}
  \includegraphics[width=\textwidth,height=6.3cm]{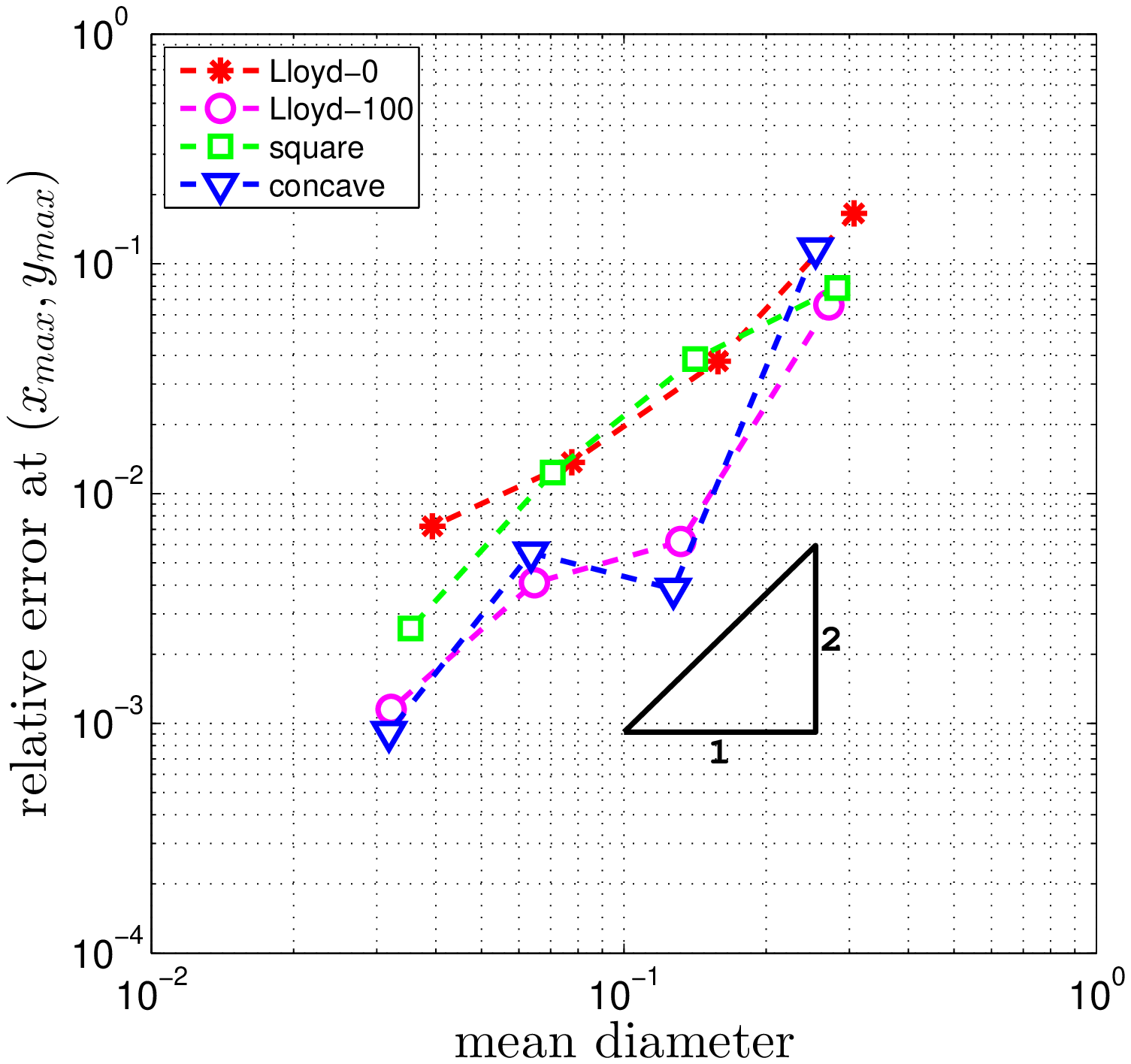}
  \end{center}
  \caption{$k=1$, relative error at $(\xmax,\ymax)$}
 \label{fig:errp-k=1-P0km1_nabla}
 \end{minipage}
\hfill
 \begin{minipage}[b]{0.49\textwidth}
  \begin{center}
  \includegraphics[width=\textwidth]{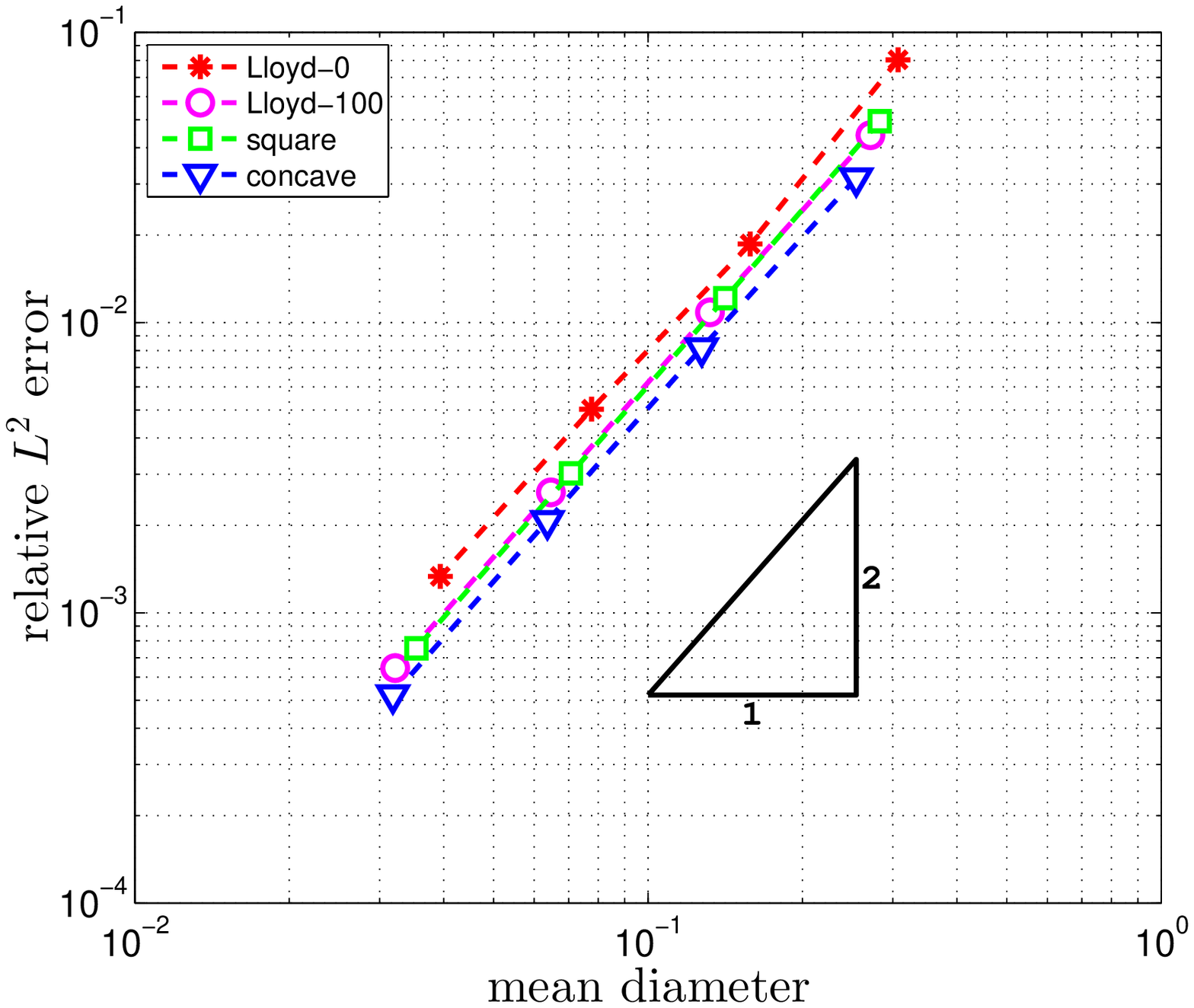}
  \end{center}
  \caption{$k=1$, relative $L^2$ error for method \eqref{simple-minded}}
 \label{fig:errL2-k=1-nabla_PNk}
 \end{minipage}
\hfill
\end{figure}

\subsection{Case $k=4$}
We show the convergence results for $k=4$; we proceed as done in the case $k=1$.
In Figs.~\ref{fig:errL2-k=4-P0km1_nabla} and \ref{fig:errH1-k=4-P0km1_nabla}
we report the relative error in $L^2$ and in $H^1$, respectively, on the four mesh sequences.
In Fig.~\ref{fig:errp-k=4-P0km1_nabla} we report the relative error at the maximum
point $(\xmax,\ymax)$.
The last figure (Fig.~\ref{fig:errL2-k=4-nabla_PNk}) shows the relative error
in $L^2$ obtained
with the method \eqref{simple-minded}. As announced, a heavy loss in the order of convergence is produced.
\begin{figure}
\hfill
 \begin{minipage}[b]{0.49\textwidth}
  \begin{center}
  \includegraphics[width=\textwidth]{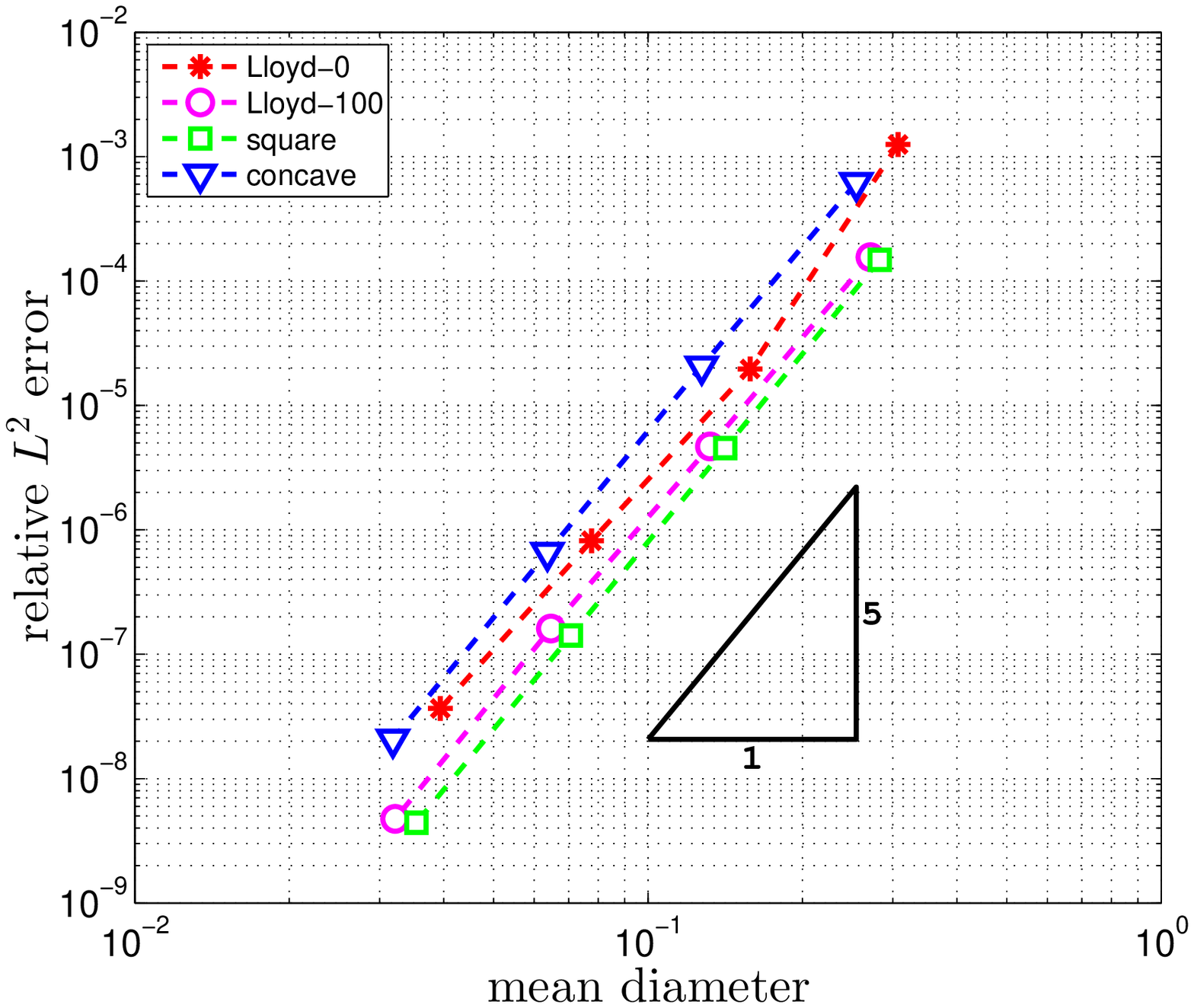}
  \end{center}
  \caption{$k=4$, relative $L^2$ error}
 \label{fig:errL2-k=4-P0km1_nabla}
 \end{minipage}
\hfill
 \begin{minipage}[b]{0.49\textwidth}
  \begin{center}
  \includegraphics[width=\textwidth]{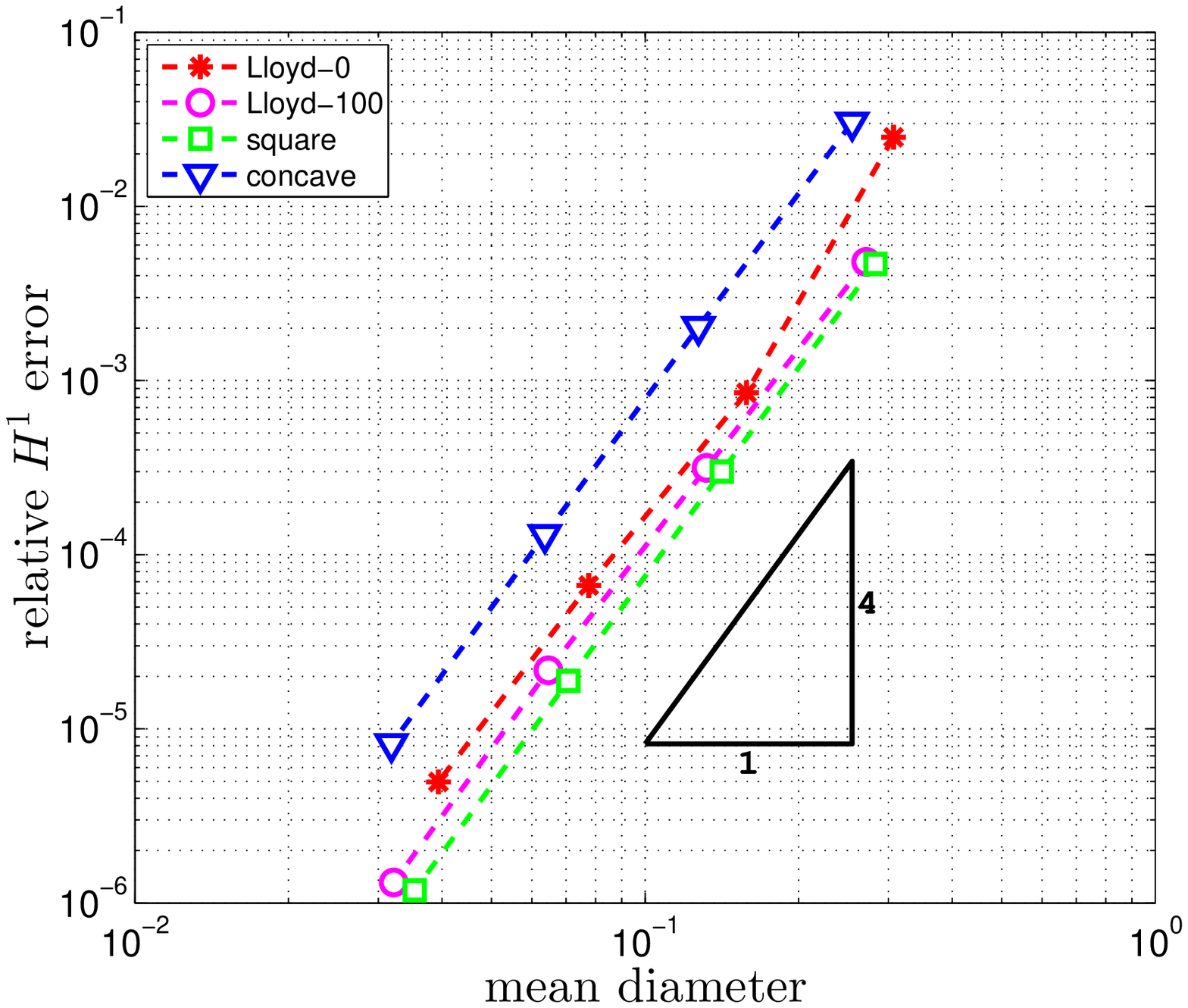}
  \end{center}
  \caption{$k=4$, relative $H^1$ error}
 \label{fig:errH1-k=4-P0km1_nabla}
 \end{minipage}
\hfill
\end{figure}
\begin{figure}
\hfill
 \begin{minipage}[b]{0.49\textwidth}
  \begin{center}
  \includegraphics[width=\textwidth,height=6.3cm]{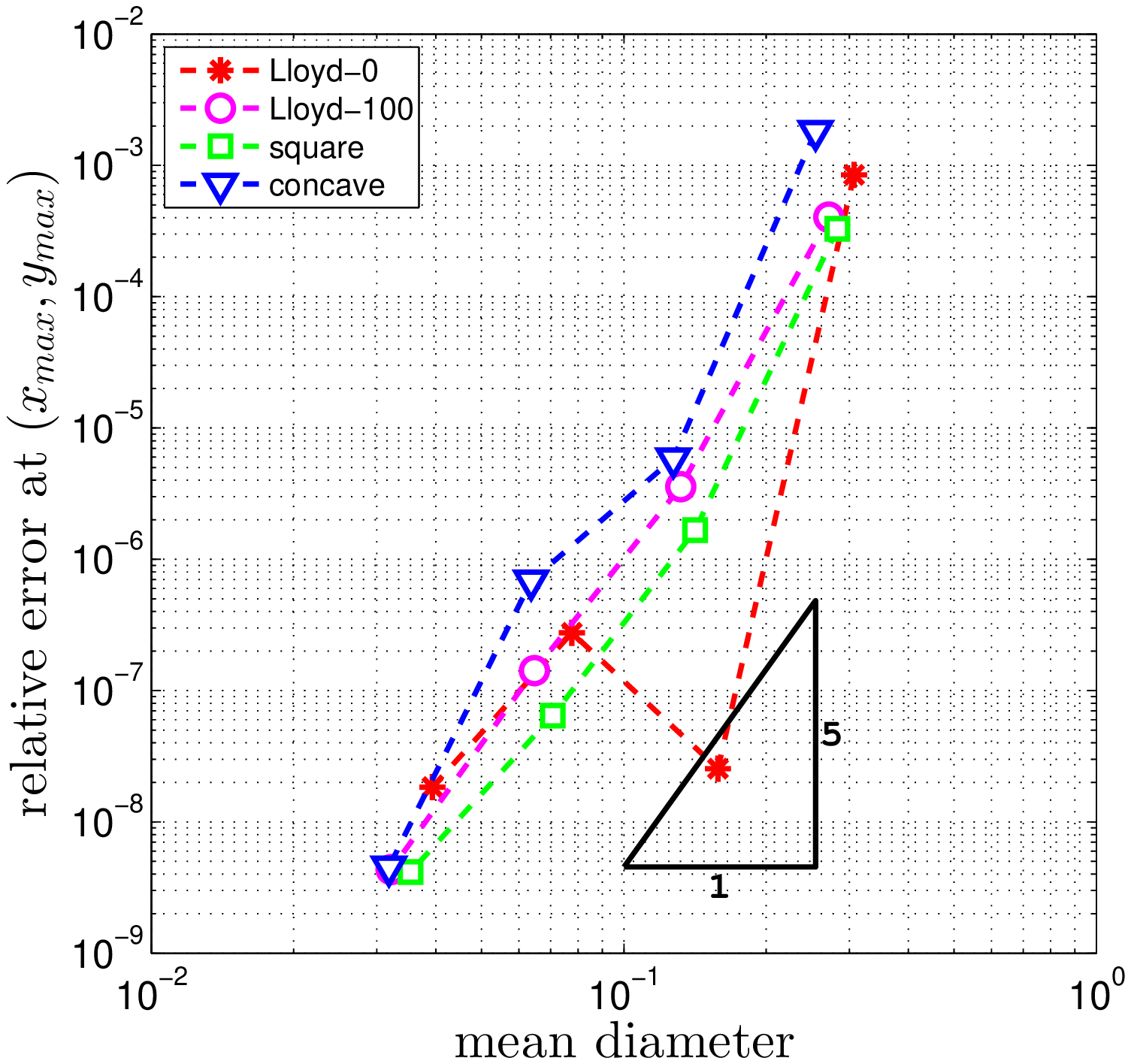}
  \end{center}
  \caption{$k=4$, relative error at $(\xmax,\ymax)$}
 \label{fig:errp-k=4-P0km1_nabla}
 \end{minipage}
\hfill
 \begin{minipage}[b]{0.49\textwidth}
  \begin{center}
  \includegraphics[width=\textwidth]{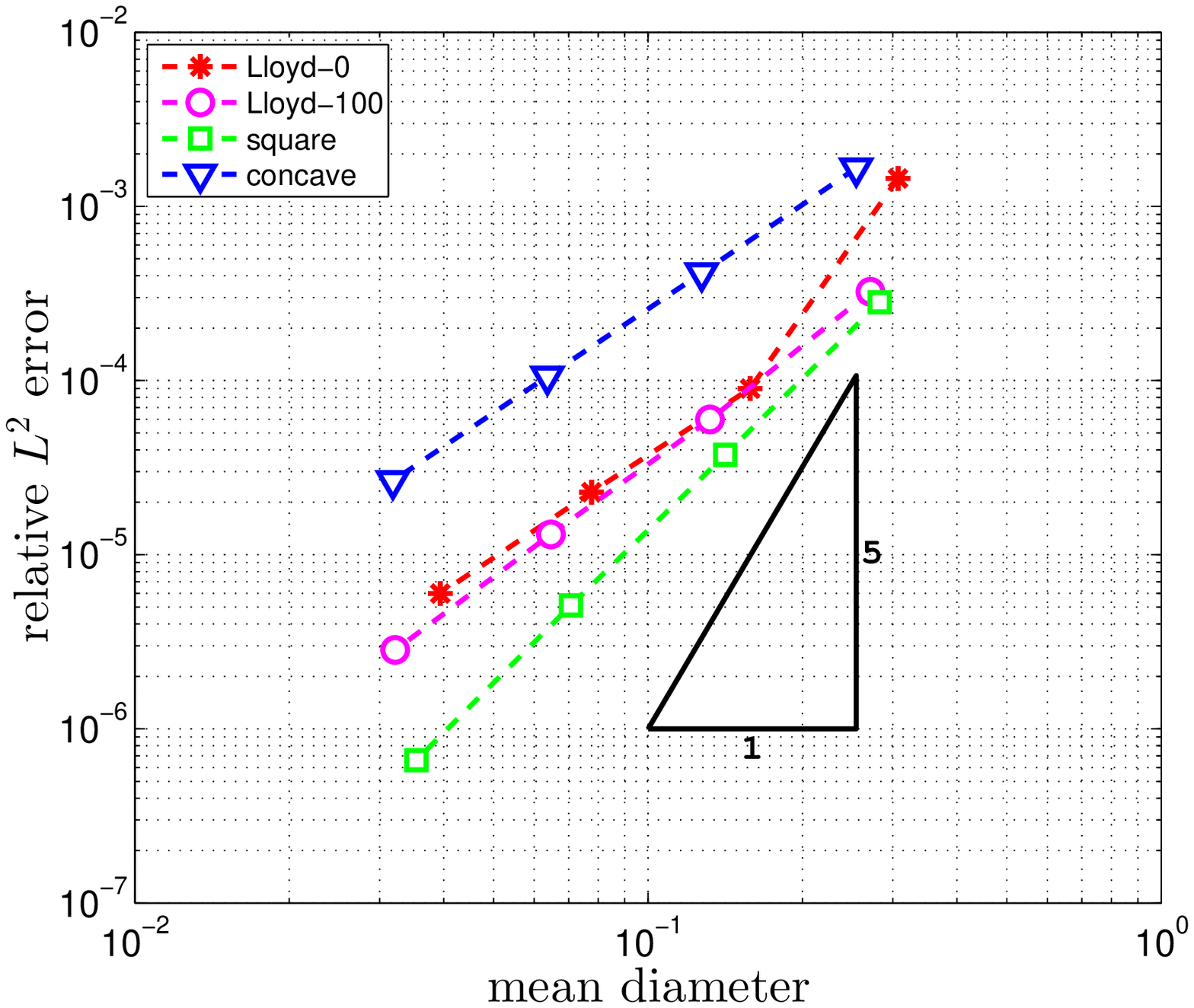}
  \end{center}
  \caption{$k=4$, relative $L^2$ error for the method \eqref{simple-minded}}
 \label{fig:errL2-k=4-nabla_PNk}
 \end{minipage}
\hfill
\end{figure}

\bigskip
We conclude that the Virtual Element Method behaves as expected and shows a remarkable
stability with respect to the shape of the mesh polygons.



\providecommand{\bysame}{\leavevmode\hbox to3em{\hrulefill}\thinspace}


\end{document}